\documentclass[a4paper,10pt]{amsart}
\usepackage[english]{babel}
\usepackage{amsmath}
\usepackage{amssymb}
\usepackage{tikz}
\usepackage{mathrsfs}
\usepackage{enumerate}
\usepackage{hyperref}

\newtheorem{thm}{Theorem}[section]
\newtheorem{Lemma}[thm]{Lemma}
\newtheorem{Proposition}[thm]{Proposition}
\newtheorem{Corollary}[thm]{Corollary}

\newtheorem*{thm*}{Theorem}

\theoremstyle{definition}

\newtheorem{Definition}[thm]{Definition}
\newtheorem{Remark}[thm]{Remark}

\definecolor{wwwwww}{rgb}{0.4,0.4,0.4}

\renewcommand{\P}{\mathbb{P}}
\newcommand{\F}{\mathbb{F}}

\newcommand{\G}{\mathbb{G}}

\newcommand{\p}{\mathbb{P}}
\newcommand{\C}{\mathbb{C}}

\DeclareMathOperator{\expdim}{expdim}

\DeclareMathOperator{\Sec}{Sec}

\DeclareMathOperator{\expd}{expdim}

\hypersetup{pdfpagemode=UseNone}
\hypersetup{pdfstartview=FitH}
		
\begin{document}
\title{On secant dimensions and identifiability of flag varieties}

\author[Ageu Barbosa Freire]{Ageu Barbosa Freire}
\address{\sc Ageu Barbosa Freire\\
Instituto de Matem\'atica e Estat\'istica, Universidade Federal Fluminense, Campus Gragoat\'a, Rua Alexandre Moura 8 - S\~ao Domingos\\
24210-200 Niter\'oi, Rio de Janeiro\\ Brazil}
\email{ageufreire@id.uff.br}

\author[Alex Casarotti]{Alex Casarotti}
\address{\sc Alex Casarotti\\ Dipartimento di Matematica e Informatica, Universit\`a di Ferrara, Via Machiavelli 30, 44121 Ferrara, Italy}
\email{csrlxa@unife.it}

\author[Alex Massarenti]{Alex Massarenti}
\address{\sc Alex Massarenti\\ Dipartimento di Matematica e Informatica, Universit\`a di Ferrara, Via Machiavelli 30, 44121 Ferrara, Italy}
\email{alex.massarenti@unife.it}

\date{\today}

\begin{abstract}
We investigate the secant dimensions and the identifiablity of flag varieties parametrizing flag of sub vector spaces of a fixed vector space. We give numerical conditions ensuring that secant varieties of flag varieties have the expected dimension, and that a general point on these secant varieties is identifiable.
\end{abstract}

\maketitle
\tableofcontents

\section{Introduction}
In the most general contest, a flag variety is a projective variety homogeneous under a complex linear algebraic group. Flag varieties play a central role in algebraic geometry, combinatorics, and representation theory \cite{Br05, BL18}.

Fix a vector space $V\cong\C^{n+1}$, over an algebraically closed field $K$ of characteristic zero, and integers $k_1\leq\ldots\leq k_r$. Let $\mathbb{G}(k_i,n)\subset\mathbb{P}^{N_i}$, where $N_i={{n+1}\choose{k_i+1}}-1$, be the Grassmannians of $k_i$-dimensional linear subspace of $\mathbb{P}(V)$ in its Pl\"ucker embedding. We have an embedding of the product of these Grassmannians 
$$\mathbb{G}(k_1,n)\times\dots\times\mathbb{G}(k_r,n)\subset\mathbb{P}^{N_1}\times\dots\times\mathbb{P}^{N_r}\subset\mathbb{P}^N$$
where $N={{n+1}\choose{k_1+1}}\cdots{{n+1}\choose{k_r+1}}-1$.

The flag variety $\F(k_1,\ldots,k_r;n)$ is the set of flags, that is nested subspaces,
$V_{k_1}\subset\cdots\subset V_{k_r}\subsetneq V$. This is a subvariety of the product of Grassmannian $\prod_{i=1}^r\G(k_i,n)$. Hence, via a product of Pl\"ucker embeddings followed by a Segre embedding we can embed $\F(k_1,\ldots,k_r;n)$ 
$$\F(k_1,\ldots,k_r;n)\hookrightarrow \mathbb{P}^{N_1}\times\dots\times\mathbb{P}^{N_r}\hookrightarrow\mathbb{P}^N$$

Consider natural numbers $a_1,\ldots,a_{n}$ such that $a_{k_1+1}=\cdots=a_{k_r+1}=1$ and $a_i=0$ for all $i\notin\{k_1+1,\ldots,k_r+1\}$. Then, $\F(k_1,\ldots,k_r;n)$ generates the subspace 
$$\mathbb{P}(\Gamma_{a_1,\ldots,a_{n}})\subseteq \mathbb{P}\left(\bigwedge^{k_1+1}V\otimes\cdots\otimes\bigwedge^{k_r+1}V\right)\subseteq \mathbb{P}^N$$
where $\Gamma_{a_1,\ldots,a_{n}}$ is the irreducible representation of $\mathfrak{sl}_{n+1} \C$ with highest weight $(a_1+\cdots+a_{n})L_1+\cdots+a_{n}L_{n}$, and $L_1+\dots + L_k$ is the highest weight of the irreducible representation $\bigwedge^{k}V$. We will denote $\Gamma_{a_1,\ldots,a_{n}}$ simply by $\Gamma_a$. By the Weyl character formula we have that 
$$\dim \mathbb{P}(\Gamma_{a}) = \prod_{1\leq i < j\leq n+1}\frac{(a_i+\dots+a_{j-1})+j-i}{j-i}-1$$
Furthermore, $\dim \F(k_1,\dots,k_r;n) = (k_1+1)(n-k_1)+\sum_{j=2}^i(n-k_j)(k_j-k_{j-1})$ and $\F(k_1,\ldots,k_r;n)=\mathbb{P}(\Gamma_a)\cap \prod_{i=1}^r\G(k_i,n)\subset \P^N$.

The geometry of these varieties has been investigated mostly from the point of view of Schubert calculus \cite{Br05} and dual defectivity \cite{Te05}. Secant varieties of small dimensional flag varieties have been studied in \cite{BD10} by taking advantage of the tropical approach to secant dimensions introduced by J. Draisma in \cite{Dr08}.

The \textit{$h$-secant variety} $\mathbb{S}ec_{h}(X)$ of a non-degenerate $n$-dimensional variety $X\subset\mathbb{P}^N$ is the Zariski closure of the union of all linear spaces spanned by collections of $h$ points of $X$. The \textit{expected dimension} of $\mathbb{S}ec_{h}(X)$ is $\expdim(\mathbb{S}ec_{h}(X)):= \min\{nh+h-1,N\}$. In general, the actual dimension of $\mathbb{S}ec_{h}(X)$ may be smaller than the expected one. In this case, following \cite[Section 2]{CC10} we say that $X$ is \textit{$h$-defective} and the number $\delta_h(X) = \expdim(\mathbb{S}ec_{h}(X)) - \dim(\mathbb{S}ec_{h}(X))$ is called the $h$-secant defect of $X$. 

We investigate secant defectivity of flag varieties following the machinery introduced in \cite{MR19}, which we now outline. Given general points $x_1,\dots,x_h\in X\subset\mathbb{P}^N$, consider the linear projection $\tau_{X,h}:X\subseteq\mathbb{P}^N\dasharrow\mathbb{P}^{N_h}$, with center $\left\langle T_{x_1}X,\dots,T_{x_h}X\right\rangle$, where $N_h:=N-1-\dim (\left\langle T_{x_1}X,\dots,T_{x_h}X\right\rangle)$. \cite[Proposition 3.5]{CC02} yields that if $\tau_{X,h}$ is generically finite then $X$ is not $(h+1)$-defective. Given $p_1,\dots, p_l\in X$ general points, we consider the linear projection $\Pi_{T^{k_1,\dots,k_l}_{p_1,\dots,p_l}}:X\subset\mathbb{P}^N\dasharrow\mathbb{P}^{N_{k_1,\dots,k_l}}$ with center the span $\left\langle T_{p_1}^{k_1}X,\dots, T_{p_l}^{k_l}X\right\rangle$ of higher order osculating spaces. We can degenerate, under suitable conditions, the linear span of several tangent spaces $T_{x_i}X$ into a subspace contained in a single osculating space $T_p^{k}X$. 
So the tangential projection $\tau_{X,h}$ degenerates to a linear projection with center contained in $\left\langle T_{p_1}^{k_1}X,\dots, T_{p_l}^{k_l}X\right\rangle$. If $\Pi_{T^{k_1,\dots,k_l}_{p_1,\dots,p_l}}$ is generically finite, then $\tau_{X,h}$ is generically finite as well, and we conclude that $X$ is not $(h+1)$-defective. In this paper we apply this strategy to flag varieties. We would like to stress that this approach, as the one introduced in \cite{Dr08}, depends heavily on an explicit parametrization of $X$. This method was successfully applied to other classes of homogeneous varieties such as Grassmannians \cite{MR19}, Segre-Veronese varieties \cite{AMR17}, Lagrangian Grassmannians and Spinor varieties \cite{FMR18}. However, its application to flag varieties involves much more difficult computations compared with the case of the Grassmannians, this is particularly reflected in Section \ref{sec3} where we introduce submersions of flag varieties into product of Grassmannians in order to study the relation among their higher osculating spaces.   

Furthermore, our results on secant defectivity, combined with a recent result in \cite{CM19}, allow us to produce a bound for identifiability of flag varieties. Recall that, given a non-degenerated variety $X\subset\mathbb{P}^N$, we say that a point $p\in\mathbb{P}^N$ is $h$-identifiable if it lies on a unique $(h-1)$-plane in $\mathbb{P}^N$ that is $h$-secant to $X$. Especially when $\mathbb{P}^N$ can be interpreted as a tensor space, identifiablity and tensor decomposition algorithms are central in applications for instance in biology, Blind Signal Separation, data compression algorithms, analysis of mixture models psycho-metrics, chemometrics, signal processing, numerical linear algebra, computer vision, numerical analysis, neuroscience and graph analysis \cite{DD13a}, \cite{DD13b}, \cite{DD15}, \cite{KAD11}, \cite{SB00}, \cite{BK09}, \cite{CGLM08}, \cite{LO15}, \cite{MR13}. Our main results in Theorem \ref{main} and Corollary \ref{CorId} can be summarized in the following statement.
\begin{thm}
Consider a flag variety $\F(k_1,\ldots,k_r;n)$. Assume that $n\geq 2k_{j}+1$ for some index $j$ and let $l$ be the maximum among these j's. Then, for 
$$h\leq\left(\frac{n+1}{k_l+1}\right)^{\lfloor \log_2(\sum_{j=1}^l k_j+l-1)\rfloor}$$
$\F(k_1,\ldots,k_r;n)$ is not $(h+1)$-defective. Furthermore, under the same bound, the general point of the $h$-secant variety of $\F(k_1,\ldots,k_r;n)$ is $h$-identifiable. 
\end{thm}
The paper is organized as follows: in Section \ref{sec1} we study higher order osculating spaces of products of Grassmannians and the linear projections from them, in Section \ref{sec2} we apply the method introduced in \cite{MR19} to products of Grassmannians, in Section \ref{sec3} we get bounds for non-secant defectivity and identifiablity of flag varieties, and in Section \ref{sec4} we investigate the variety of secant lines of spacial flag varieties of type $\mathbb{F}(0,k;n)$.

\subsection*{Acknowledgments}
The first named author would like to thank FAPERJ and Massimiliano Mella (PRIN $2015$, Geometry of Algebraic Varieties, 2015EYPTSB-005) for the financial support, and the University of Ferrara for the hospitality during the period in which the majority of this work was completed. 
The third named author is a member of the Gruppo Nazionale per le Strutture Algebriche, Geometriche e le loro Applicazioni of the Istituto Nazionale di Alta Matematica F. Severi (GNSAGA-INDAM).

\section{Higher osculating behavior of products of Grassmannians}\label{sec1}
Consider the product $\mathbb{G}(k_1,n)\times\dots\times\mathbb{G}(k_r,n)\subset\mathbb{P}^{N_1}\times\dots\times\mathbb{P}^{N_r}\subset\mathbb{P}^N$, and given a non-negative integer $k$ define
$$\Lambda_k=\{I\subset\{0,\ldots,n\}\:|\:|I|=k+1\}$$
For any $I=\{i_0,\ldots,i_k\}\in\Lambda_k$ let $e_{I}\in\mathbb{G}(k,n)$ be the point corresponding to $e_{i_0}\wedge\cdots\wedge e_{i_k}\in\bigwedge^{k+1}\C^{n+1}$. We will denote by $Z_I$ the Pl\"ucker coordinates on $\mathbb{P}(\bigwedge^{k+1}\C^{n+1})$. 

From \cite{MR19} we have a notion of distance in $\Lambda_k$ given by
\stepcounter{thm}
\begin{equation}\label{Distance}
d(I,J)=|I|-|I\cap J|
\end{equation}
for all $I,J\in\Lambda_k$. More generally, we define 
$$\Lambda=\Lambda_{k_1}\times\dots\times\Lambda_{k_r}$$
Given $I=\{I^1,\ldots,I^r\}\in \Lambda$ let $e_I\in \prod_{i=1}^r\mathbb{G}(k_i,n)$ be the point corresponding to
$e_{I^1}\otimes\dots\otimes e_{I^r}\in\mathbb{P}^N$, and by $Z_I$ the corresponding homogeneous coordinate of $\mathbb{P}^N$. Furthermore, for all $I,J\in\Lambda$ with $I=\{I^1,\ldots,I^r\}$ and $J=\{J^1,\ldots,J^r\}$, we define their distance as
$$d(I,J)=\sum_{i=1}^rd(I^i,J^i)$$
where $d(I^i,J^i)$ is the distance defined in (\ref{Distance}).

From now on we will assume that $n\geq 2k_r+1$. Under this assumption $\Lambda$ has diameter $r+\sum_{i=1}^r k_i$ with respect to this distance.

In the following, we give an explicit description of the osculating spaces of $\prod_{i=1}^r\mathbb{G}(k_i,n)$ at coordinate points. 

\begin{Proposition}\label{Osc_Product}
For each $s\geq 0$ 
$$
T_{e_I}^s\left(\prod_{i=1}^r\mathbb{G}(k_i,n)\right) = \langle e_J\: ; \:d(I,J)\leq s\rangle = \{ Z_J=0\: ; \:d(I,J)> s\}\subset \mathbb{P}^N
$$
In particular, $T_{e_I}^s\left(\prod_{i=1}^r\G(k_i,n)\right)= \mathbb{P}^N$ for $s\geq r+\sum_{i=1}^rk_i$.
\end{Proposition}
\begin{proof}
Set $I=\{I^1,\dots,I^r\}\in \Lambda$. We may assume that $I^i=\{0,\dots,k_i\}$ for each $1\leq i\leq r$ and consider the following parametrization of $\prod_{i=1}^r\mathbb{G}(k_i,n)$ in a neighborhood of $e_I$:
\stepcounter{thm}
\begin{equation}\label{Param_Product}
\begin{array}{ccccl}
\varphi&:&\prod_{i=1}^r\C^{(k_i+1)(n-k_i)}&\longrightarrow&\mathbb{P}^N\\
&&\left[I_{k_i+1}\:\:\:(x_{l,m}^i)\right]_{i=1,\dots,r}&\longmapsto&\left(\prod_{i=1}^r\det(M_{J^i})\right)_{J=\{J^1,\ldots,J^r\}\in \Lambda}
\end{array}
\end{equation}
where $M_{J^i}$ is the submatrix obtained from $\left[I_{k_i+1},\:(x_{l,m}^i)_{\substack{0\leq l\leq k_i\\k_i+1\leq m\leq n}}\right]$ by considering the columns indexed by $J^i$.

For each $J\in \Lambda$, we will denote $\prod_{i=1}^r\det(M_{J^l})$ simply by $\det(M_J)$. Note that each variable appears in degree at most one in the coordinates of $\varphi$. Therefore, deriving two times with respect to the same variable always gives zero. Furthermore, as $\det(M_{J})$ has degree at most $r+\sum_{i=1}^r k_i$ all partial derivatives of order greater or equal than $r+\sum_{i=1}^r k_i$ are zero. Thus, it is enough to prove the claim for $s \leq r+\sum_{i=1}^r k_i$.

Given $J=\{J^1,\ldots,J^r\}\in \Lambda$, let $i,k,k'$ be integers such that $1\leq i\leq r$, $k\in\{0\ldots,k_i\}$ and $k'\in\{k_i+1,\ldots,n\}$. Then
$$
\dfrac{\partial \det(M_{J})}{\partial  x_{k,k'}^i}=\left\{\begin{array}{cl}
 \pm \dots\det(M_{J^{i-1}})\det(M_{J^i,k,k'})\det(M_{J^{i+1}})\dots& k'\in J^i\\
 0&k'\notin J^i
\end{array}\right.
$$
where $M_{J^i,k,k'}$ is the submatrix obtained from $M_{J^i}$ by deleting the column indexed by $k'$ and the row indexed by $k$.

More generally, let $m_1,\ldots,m_r$ be non-negative integers such that their sum is bigger than one. For each $i=1,\ldots,r$ consider 
$$K_i=\{k_1^i,\ldots,k_{m_i}^i\}\subset \{0,\ldots,k_i\}\:\text{ and }\:K_i'=\{k_1'^i,\ldots,k_{m_i}'^i\}\subset\{k_i+1,\ldots,n\}$$
with $|K_i|=|K'_i|=m_i$. Now, set $m=m_1+\cdots+m_r$ and
$$K=\{K_1,\ldots,K_r\}, \:K'=\{K_1',\ldots,K_r'\}$$
Therefore, denoting $\partial x_{k^1_1,k_1'^1}^{1}\cdots \partial x_{k^r_{m_r},k'^r_{m_r}}^{r}$ simply by $\partial^m K,K'$ we have 
$$
\dfrac{\partial ^m \det(M_J)}{\partial^m K,K'}=\left\{\begin{array}{cl}
\pm \prod_{i=1}^r\det(M_{J^{i},K_i,K_i'})&\text{ if }K'\subset J \: \:\text{ and }m\leq d(I,J)=\deg (\det (M_J))\\
0&\text{ otherwise }
\end{array}\right.$$
for any $J \in \Lambda$, where $K'\subset J$ means that $\{k_1'^1,\ldots,k_{m_1}'^1\}\subset J^{1},\ldots,\{k_1'^r,\ldots,k_{m_r}'^r\}\subset J^{r}$, and $M_{J^{i},K_i,K_i'}$ is the submatrix obtained from $M_{J^{i}}$ deleting the columns indexed by $K_i'$ and the rows indexed by $K_i$. Thus,
$$
\dfrac{\partial ^m \det(M_J)}{\partial^m K,K'}(0)=\left\{\begin{array}{cl}
\pm1&\text{ if }J^{i}=K_i'\cup (\{I^{i}\setminus K_i\}) \text{ for each }i=1,\ldots,r\\
0&\text{ otherwise }
\end{array}\right.$$
Finally, let us denote by $J=K'\cup \{I\setminus K\}$  the element in $\Lambda$ for which $J^{i}=K_i'\cup (\{I^{i}\setminus K_i\})$ for each $i=1,\ldots,r$. Then, we have that
$$
\dfrac{\partial ^m \varphi}{\partial^m K,K'}(0)=
\pm e_{K'\cup (\{I\setminus K\})}$$
Note that $d(I,K'\cup \{I\setminus K\})=m$, and any $J\in \Lambda$ with $d(I,J)=m$ may be written as $K'\cup \{I\setminus K\}$.
Thus, we get that 
$$\left\langle\begin{array}{c}
\dfrac{\partial^{m} \varphi}{\partial^{m} K,K'}(0)\:|\:m\leq s\\
\end{array}
\right\rangle=\langle e_J\:|\:d(I,J)\leq s\rangle
$$
which proves the claim.
\end{proof}

Now, it is immediate to compute the dimension of the osculating spaces of $\prod_{i=1}^r\mathbb{G}(k_i,n)$.

\begin{Corollary}\label{Dim_Osc_Product}
For any point $p\in \prod_{i=1}^r\mathbb{G}(k_i,n)$ we have
$$
\begin{array}{ccl}
\dim T_{p}^s(\prod_{i=1}^r\mathbb{G}(k_i,n))&=&\displaystyle\sum_{\substack{i=1,\ldots,r\\0\leq s_i\leq k_i+1,\\s_1+\cdots+s_r\leq s}}{{n-k_1}\choose{s_1}}{{k_1+1}\choose{s_1}}\cdots{{n-k_r}\choose{s_r}}{{k_r+1}\choose{s_r}}
\end{array}$$
for any $0\leq s< r+\sum_{i=1}^r k_i$ while $T_{p}^s\left(\prod_{i=1}^r\mathbb{G}(k_i,n)\right)=\mathbb{P}^N$ for any $s\geq r+\sum_{i=1}^r k_i$.
\end{Corollary}
\begin{proof}
Since the general linear group $GL(n+1)$ acts transitively on $\prod_{i=1}^r\mathbb{G}(k_i,n)$ the statement follows from Proposition \ref{Osc_Product}.
\end{proof}

\subsection{Osculating Projections}
For a general point $p\in\prod_{i=1}^r\G(k_i,n)$, we will denote $T_p^s\left(\prod_{i=1}^r\G(k_i,n)\right)$ simply by $T_p^s$. Now, take $0\leq s\leq r+\sum_{i=1}^rk_i$ and $I\in \Lambda$. By Proposition \ref{Osc_Product} the linear projection of $\prod_{i=1}^r\G(k_i,n)$ from $T_{e_I}^s$ is given by
$$
\begin{array}{cccl}
\Pi_{T_{e_{I}}^s}:&\prod_{i=1}^r\G(k_i)&\dashrightarrow&\p^{N'_s}\\
&(Z_J)_{J\in\Lambda}&\longmapsto&(Z_J)_{J\in\Lambda\:|\:d(I,J)>s}
\end{array}.
$$
Moreover, given $I'\subset\{0,\dots,n\}$ with $|I'|=m$ we have the linear projection 
$$
\begin{array}{ccccl}
\pi_{I'}&:&\p^n&\dashrightarrow&\p^{n-m}\\
&&(x_i)&\longmapsto&(x_i)_{i\in\{0,\ldots,n\}\setminus I'}
\end{array}
$$
which in turns induces the linear projection 
$$
\begin{array}{ccccl}
\Pi_{I'}&:&\G(k,n)&\dashrightarrow&\mathbb{G}(k,n-m)\\
&&V&\longmapsto&\langle\pi_{I'}(V)\rangle\\
&&(Z_J)_{J\in\Lambda_k}&\longmapsto&(Z_J)_{J\in\Lambda_k\:|\:J\cap I'=\emptyset}
\end{array}
$$
whenever $k<n-m$.

Finally, let us fix $I=\{I^1,\dots,I^r\}\in \Lambda$ and take $m_1,\dots,m_r$ integers such that $m_i\leq k_i+1$ for each $i=1,\dots,r$. Then, given $I'^1\subset I^1,\dots,I'^r\subset I^r$, with $|I'^i|=m_i$, we have a projection
$$
\begin{array}{ccccl}
\prod_{i=1}^r\Pi_{I'^i}&:&\prod_{i=1}^r\G(k_i,n)&\dashrightarrow&\prod_{i=1}^r\G(k_i,n-m_i)\\
&&V_1\times\dots\times V_r&\longmapsto&\Pi_{I'^1}(V_1)\times\dots\times\Pi_{I'^r}(V_r)
\end{array}.
$$
Note that a general fiber of $\prod_{i=1}^r\Pi_{I'^i}$ is isomorphic to $\prod_{i=1}^r\G(k_i,k_i+m_i)$. Indeed, let $x=\prod_{i=1}^r\Pi_{I'^i}\left((V_i)_{i=1}^r\right)\in\prod_{i=1}^r\G(k_i,n-m_l)$ be a general point. Then, we have
\begin{center}
$\overline{\left(\prod_{i=1}^r\Pi_{I'^i}\right)^{-1}(x)}=\left\{(W_i)_{i=1}^r\in\prod_{i=1}^r\G(k_i,n)\:|\:W_i\subset\langle V_i,e_{j_1^i},\ldots,e_{j_{m_i}^i}\rangle,\:i=1,\ldots,r\right\}$.
\end{center}

\begin{Lemma}\label{Lemma1}
Let us fix $I=\{I^1,\ldots,I^r\}\in\Lambda$. If $0\leq s\leq r-2+\sum_{i=1}^r k_i$ and $I'^i\subset I^i$ with $|I'^i|=m_i$ for each $i=1,\ldots,r$, then the rational map $\Pi_{T_{e_{I}}^s}$ factors through $\prod_{i=1}^r\Pi_{I'^i}$ whenever $\sum_{i=1}^r m_i= s+1
$.
\end{Lemma}
\begin{proof}
Since the diameter of $\Lambda$ is $r+\sum k_i$ we have $\{J\in \Lambda\:|\:d(I,J)\leq s\}\subsetneq \Lambda$ and then $\Pi_{T_{e_{I}}^s}$ is well-defined.

On the other hand, if $J=\{J^1,\ldots,J^r\}\in\Lambda$ is such that $J^i\cap I'^i=\emptyset$ for all $i=1,\ldots,r$, then $d(I,J)\geq\sum_{i=1}^r m_i >s$ which yields that the center of  $\Pi_{T_{e_{I}}^s}$ is contained in the center of $\prod_{i=1}^r\Pi_{I'^i}$.
\end{proof}

\begin{Proposition}\label{Prop1}
The rational map $\Pi_{T_{e_I}^s}$ is birational for all $0\leq s\leq r-2+\sum_{i=1}^r k_i$.
\end{Proposition}
\begin{proof}
Since $T_{e_{I}}^s$ contains $T_{e_{I}}^{s-1}$ it is enough to prove the statement for $s=r-2+\sum_{i=1}^r k_i$. Let us fix $m\in\{1,\ldots,r\}$. By Lemma \ref{Lemma1}, for each subset $I'^m\subset I^m$ with $|I'^m|=k_m$ there is a rational map $\pi_{I'^m}$ that makes the following diagram commutative. 
\[
\begin{tikzpicture}[xscale=3.9,yscale=-1.8]
    \node (A0_0) at (0, 0) {$\prod_{i=1}^r\G(k_i,n)$};
    \node (A0_1) at (1, 0) {$\P^{N^{'}_{s}}$};
    \node (A1_1) at (1, 1) {$\left(\prod_{i\neq m}\G(k_i,n-k_i-1)\right)\times\G(k_m,n-k_m)$};
    \path (A0_0) edge [->,dashed]node [auto] {$\scriptstyle{\Pi_{T_{e_{I}}^s}}$} (A0_1);
    \path (A0_1) edge [->,dashed]node [auto] {$\scriptstyle{\pi_{I^{'m}}}$} (A1_1);
    \path (A0_0) edge [->,dashed,swap]node [auto] {$\scriptstyle{\left(\prod_{i\neq m}\Pi_{I^i}\right)\times \Pi_{I^{'m}}}$} (A1_1);
\end{tikzpicture}
\]
Let $x=\Pi_{T_{e_{I}}^s}(\{V_i\}_{i=1}^r)$ be a general point and $X\subset \prod_{i=1}^r\G(k_i,n)$ be the fiber of $\Pi_{T_{e_{I}}^s}$ over $x$. Set $x_{I'^m}=\pi_{I'^m}(x)$, and denote by $X_{I'^m}\subset \prod_{i=1}^r\G(k_i,n)$ the fiber of $\left(\prod_{i\neq m}\Pi_{I^i}\right)\times \Pi_{I'^m}$ over $x_{I'^m}$. Thus,
$$X\subset \displaystyle\bigcap_{I'^m} X_{I'^m}$$
where this intersection runs over all $I'^m\subset I^m$ with $|I'^m|=k_m$ and $m=1,\ldots,r$. Now, if $(W_i)_{i=1}^r$ is a general point in $ X$ then 
$$W_m\subset \langle e_{j_1},\ldots,e_{j_{k_m}},V_m\rangle\:\text{ for any }\:I'^m=\{e_{j_1},\ldots,e_{j_{k_m}}\}\subset I^m$$
Therefore,
$$W_m\subset \displaystyle\bigcap_{I'^m}\langle e_{j_1},\dots,e_{j_{k_m}},V_m\rangle=V_m$$
This implies $W_m=V_m$ for every $m=1,\ldots,r$. Since we are working in characteristic zero, we conclude that $\Pi_{T_{e_{I}}^s}$ is birational.
\end{proof}
The next step is to study linear projections from the span of several osculating spaces. In particular, we want to understand when such a projection is birational. First of all, note that the order of osculating spaces can not exceed $r-2+\sum_{i=1}^r k_i$. Furthermore, in order to carry out the computations, we need to consider just the coordinates points of $\prod_{i=1}^r\G(k_i,n)$ such that the corresponding linear subspaces are linearly independent in $\C^{n+1}$, then we can use at most
$$\alpha:=\left\lfloor\dfrac{n+1}{k_r+1}\right\rfloor$$
of them. Now, let us consider the points $e_{I_1},\dots,e_{I_{\alpha}}\in\prod_{i=1}^r\G(k_i,n)$, where
\stepcounter{thm}
\begin{equation}\label{Set_Indexes}
\begin{array}{l}
I_1=\{I_1^1=\{0,\ldots,k_1\},\ldots,I_{1}^r=\{0,\ldots,k_r\}\}\\
I_2=\{I_2^1=\{k_r+1,\ldots,k_r+k_1+1\},\ldots,I_{2}^r=\{k_r+1,\ldots,k_r+k_r+1\}\}\\
\vdots\\
I_{\alpha}=\{\ldots,I_ {\alpha}^i=\{(k_r+1)(\alpha-1),\ldots,(k_r+1)(\alpha-1)+k_i\},\ldots\}
\end{array}
\end{equation}
Let $s_1,\dots,s_{\alpha}$ be integers such that $0\leq s_m\leq r-2+\sum_{i=1}^r k_i$. Denote the linear subspace $\langle T_{e_{I_1}}^{s_1},\dots,T_{e_{I_m}}^{s_m}\rangle$ simply by $T_{e_{I_1},\ldots,e_{I_m}}^{s_1,\ldots,s_m}$. Then, for $m\leq \alpha$ we have the linear projection
$$
\begin{array}{ccccl}
\Pi_{T_{e_{I_1},\ldots,e_{I_m}}^{s_1,\ldots,s_m}}&:&\prod_{i=1}^r\G(k_i,n)&\dashrightarrow&\p^{N'_{s_1,\ldots,s_m}}\\
&&(Z_J)_{J\in\Lambda}&\longmapsto&(Z_J)_{J\in\Lambda\: | \:d(J,I_1)>s_1,\ldots,d(J,I_m)>s_m}
\end{array}
$$
Now, consider $I_1,\ldots,I_{\alpha}$ as in (\ref{Set_Indexes}), and $I_m'^i\subset I_m^i$ with $|I_m'^i|=s_m^i$ for each $1\leq m\leq \alpha$ and $i=1,\ldots,r$, where $s_m^i$ are non-negative integers. If $I'^i$ denotes the union $\bigcup_{m=1}^{\alpha} I_m'^i$, then for each $i=1,\ldots,r$ we have a linear projection of $\p^n$ 
$$
\begin{array}{ccccl}
\pi_{I'^i}&:&\p^n&\dashrightarrow&\p^{n-\sum_{m=1}^{\alpha} s_m^i}\\
&&(x_i)_{0\leq i\leq n}&\longmapsto&(x_i)_{0\leq i\leq n\text{ and }i\notin I'^i}
\end{array}
$$
which in turns induces the following projection
$$
\begin{array}{ccccl}
\Pi_{I'^i}&:&\G(k_i,n)&\dashrightarrow&\mathbb{G}(k_i,n-\sum_{m=1}^{\alpha} s_m^i)\\
&&V&\longmapsto&\langle\pi_{I'^i}(V)\rangle\\
&&(Z_J)_{J\in\Lambda_{k_i}}&\longmapsto&(Z_J)_{J\in\Lambda_{k_i}\:|\:J\cap I'^i=\emptyset}
\end{array}
$$
whenever $n-\sum_{m=1}^{\alpha} s_m^i\geq k_i$. Finally, if $n-\sum_{m=1}^{\alpha} s_m^i\geq k_i$ for each $i=1,\ldots,r$, then the projections above induce a projection
$$
\begin{array}{ccccl}
\prod_{i=1}^r\Pi_{I'^i}&:&\prod_{i=1}^r\G(k_i,n)&\dashrightarrow&\prod_{i=1}^r\mathbb{G}(k_i,n-\sum_{m=1}^{\alpha} s_m^i)\\
&&(V_1,\dots,V_r)&\longmapsto&(\Pi_{{ I'^1}}(V_1),\dots,\Pi_{{I'^r}}(V_r))
\end{array}
$$
\begin{Lemma}\label{Lemma2}
Let $I_1,\ldots,I_{\alpha}$ be as in (\ref{Set_Indexes}), $m,s_1,\ldots,s_m$ integers such that $1<m\leq \alpha$ and $0\leq s_i \leq r-2+\sum_{i=1}^r k_i$. Now, consider $I_1'^i\subset I_1^i,\ldots,I_m'^i\subset I_m^i$ with $|I_j'^i|=s_j^i$, where $s_j^i$ is a non-negative integer for each $i=1,\ldots,r$ and $1\leq j\leq m$. For $j>m$ and $i=1,\ldots,r$ set $I_j'^i=\emptyset\subset I_j^i$. Denote by $I'^i$ the union $\bigcup_{j=1}^{\alpha}I_j'^i$ for each $i=1,\ldots,r$ and assume that 
\begin{itemize}
\item[(i)] $n-\sum_{j=1}^m s_j^i\geq k_i$ for each $i=1,\ldots,r$;
\item[(ii)] $\sum_{i=1}^rs_j^i\geq s_j+1$ for each $j=1,\ldots, m$.
\end{itemize} 
Then, the rational maps $\prod_{i=1}^r\Pi_{I'^i}$ and $\Pi_{T_{e_{I_1},\ldots,e_{I_m}}^{s_1,\ldots,s_m}}$ are well-defined and the former factors through the latter.
\end{Lemma}
\begin{proof}
Note that $\Pi_{T_{e_{I_1},\ldots,e_{I_m}}^{s_1,\ldots,s_m}}$ is well-defined if and only if $\{J\in\Lambda\: | \:d(J,I_1)>s_1,\ldots,d(J,I_m)>s_m\}\neq \emptyset$. From (i) we have that for each $1\leq i\leq r$ the set $\{0,\ldots,n\}\setminus I'^i$ has at least $k_i+1$ elements. Therefore, we have a set $J^i\subset\{0,\ldots,n\}\setminus I'^i$ of cardinality $k_i+1$ and taking $J=\{J^1,\ldots,J^r\}\in \Lambda$ we have 
$$d(I_j,J)=\displaystyle\sum_{i=1}^rd(I_j^i,J^i)\geq \sum_{i=1}^rs_j^i= s_j+1>s_j$$
for each $1\leq j\leq m$. Hence, $\Pi_{T_{e_{I_1},\ldots,e_{I_m}}^{s_1,\ldots,s_m}}$ is well-defined. Now, note that (i) yields that $\prod_{i=1}^r\Pi_{I'^i}$ is well-defined. Furthermore, if $J\in \Lambda$ and $J^i\cap  I'^i=\emptyset$ for all $i=1,\ldots,r$, then $d(J,I_1)>s_1,\ldots,d(J,I_m)>s_m$. Thus, the center of $\Pi_{T_{e_{I_1},\ldots,e_{I_m}}^{s_1,\ldots,s_m}}$ is contained in the center of $\prod_{i=1}^r\Pi_{I'^i}$.
\end{proof}
\begin{Proposition}\label{Prop1}
Let $I_1,\ldots,I_{\alpha-1}$ be as in (\ref{Set_Indexes}) and $s_1,\ldots,s_{\alpha-1}$ be integers such that $0\leq s_j\leq s=r-2+\sum_{i=1}^r k_i$. Then, the projection $\Pi_{T_{e_{I_1},\ldots,e_{I_{\alpha-1}}}^{\hspace{0,1 cm}s_1\hspace{0,1 cm},\ldots,\hspace{0,1 cm}s_{\alpha-1}}}$ is birational. 
\end{Proposition}
\begin{proof}
Fix $m\in\{1,\ldots,r\}$. For any $j=1,\ldots,\alpha-1$ consider $I_j'^{m}\subset I_j^{m}$ with $|I_j'^{m}|=k_{m}$ and $I_j'^i=I_j^i$ for $i\neq m$. Set $I'^i=\bigcup_{j=1}^{\alpha-1}I_j'^i$, then
$$n-(\alpha-1)(k_i+1)\geq n-(\alpha-1)(k_r+1)\geq n-\dfrac{(n-k_r)}{k_r+1}(k_r+1)\geq k_r\geq k_i$$
and
$$n-(\alpha-1)k_m\geq n-\dfrac{(n-k_r)}{k_r+1}k_r=\dfrac{nk_r+n-nk_r+k_r^2}{k_r+1}\geq\dfrac{2k_r+1+k_r^2}{k_r+1}\geq k_m+1$$
Thus, our set of subsets $I_j'^i$ satisfies (i) in Lemma \ref{Lemma2}. Furthermore, for each $j=1,\ldots,\alpha-1$
$$\displaystyle\sum_{i=1}^r |I_j'^l|=k_{m}+\displaystyle\sum_{i\neq m}(k_i+1)=r-1+\sum_{i=1}^{r}k_i=s+1$$
Therefore, by Lemma \ref{Lemma2} there exists a rational map $\pi_{I'^{m}}$ that makes the following diagram commutative
\[
  \begin{tikzpicture}[xscale=3.9,yscale=-1.8]
    \node (A0_0) at (0, 0) {$\prod_{i=1}^r\G(k_i,n)$};
    \node (A0_1) at (1, 0) {$\P^{N'_{s,\ldots,s}}$};
    \node (A1_1) at (1, 1) {$\prod_{i=1}^r\G(k_i,n-\sum_{j=1}^{\alpha-1} |I_j'^i|)$};
    \path (A0_0) edge [->,dashed]node [auto] {$\scriptstyle{\Pi_{T_{e_{I_1}},\ldots,e_{I_{\alpha-1}}}}$} (A0_1);
    \path (A0_1) edge [->,dashed]node [auto] {$\scriptstyle{\pi_{I'^m}}$} (A1_1);
    \path (A0_0) edge [->,dashed,swap]node [auto] {$\scriptstyle{\prod_{i=1}^r\Pi_{I'^i}}$} (A1_1);
  \end{tikzpicture}
\]
Now, let $x=\Pi_{T_{e_{I_1},\ldots,e_{I_{\alpha-1}}}^{\hspace{0,1 cm}s\hspace{0,1 cm},\ldots,\hspace{0,1 cm}s}}(\{V_i\}_{i=1}^r)$ be a general point in the image of $\Pi_{T_{e_{I_1},\ldots,e_{I_{\alpha-1}}}^{\hspace{0,1 cm}s\hspace{0,1 cm},\ldots,\hspace{0,1 cm}s}}$, and $X\subset \prod_{i=1}^r\G(k_i,n)$ be the fiber of $\Pi_{T_{e_{I_1},\ldots,e_{I_{\alpha-1}}}^{\hspace{0,1 cm}s\hspace{0,1 cm},\ldots,\hspace{0,1 cm}s}}$ over $x$. Set $x_{I'^m}=\pi_{I'^m}(x)$ and denote by $X_{I'^m}\subset \prod_{i=1}^r\G(k_i,n)$ the fiber of $\prod_{i=1}^r\Pi_{I'^i}$ over $x_{I'^m}$. Therefore, $X\subset \displaystyle\bigcap_{I'^m} X_{I'^m}$ where this intersection runs over all subsets $I'^m=\bigcup_{j=1}^{\alpha-1}I_j'^m$ with $I_j'^m\subset I_j^m$ and $|I_j'^m|=k_m$.
In particular, if $\{W_i\}_{i=1}^r\in X$ is a general point, then we must have $W_m\subset\langle e_i\:|\:i\in I'^m\:;\:V_m\rangle$ and hence $W_m\subset\displaystyle\bigcap_{I'^m}\langle e_i\:|\:i\in I'^m;V_m\rangle$. Now, since $|I_j'^m|=k_m$ we have $\displaystyle\bigcap_{I'^m}\langle e_i\:|\:i\in I'^m\rangle=\emptyset$ and then $V_m=\displaystyle\bigcap_{I'^m}\langle e_i\:|\:i\in I'^m\:;\:V_m\rangle$ which in turn yields $W_m=V_m$, for all $m=1,\dots,r$.
\end{proof}
Now, we want to understand what is the largest integer $s'$ for which $\Pi_{T_{e_{I_1},\ldots,e_{I_{\alpha-1}},e_{I_{\alpha}}}^{s,\ldots,s,s'}}$ is birational. 
\begin{Proposition}\label{Osc_Proj_Product}
Let $I_1,\ldots,I_{\alpha}$ be as in (\ref{Set_Indexes}) and $s=r-2+\sum_{i=1}^r k_i$. Consider $s_i'=\min\{k_i+1,n-\alpha(k_i+1)\}$ for $i\neq r$, $s_r'=\min\{k_r,n-\alpha k_r-1\}$, and set $s'=\sum_{i=1}^r s_i'-1\leq s$. Then, 
\begin{itemize}
\item[-] $\Pi_{T_{e_{I_1},\ldots,e_{I_{\alpha-1}},e_{I_{\alpha}}}^{s,\ldots,s,s'-1}}$ is birational whenever $\alpha(k_r+1)-1<n<k_r^2+3k_r+1$;
\item[-] $\Pi_{T_{e_{I_1},\ldots,e_{I_{\alpha}}}^{s,\ldots,s}}$ is birational whenever $n\geq k_r^2+3k_r+1$.
\end{itemize} 
\end{Proposition}
\begin{proof}
First, let us assume that $s_r'<k_r$, that is $n-\alpha k_r-1<k_r$, or equivalently
$$n-\alpha k_r<k_r+1\:\Leftrightarrow\:n-\dfrac{(n+1)}{k_r+1}k_r<k_r+1\:\Leftrightarrow\:n<k_r^2+3k_r+1$$
Now, fix a pair of indexes $(l,m)\in\{1,\ldots,\alpha-1\}\times\{1,\ldots,r\}$  and  consider subsets $I_j'^i\subseteq I_j^i$ with $|I_j'^i|=a_{j,i}$ for each $i\in\{1,\ldots,r\}$ and $j\in\{1,\ldots\alpha\}$ such that 
$$a_{j,i}=\left\{\begin{array}{lcl}
k_i&\text{if}& i=m,\: j=l\:\text{ or }\: i= r,\: j\neq l,\alpha;\\
k_i+1&\text{if}& i=r\neq m,\: j=l\:\text{ or }\: i\neq m,r\:\text{ or }\:l\neq m,\alpha;\\
s_i'&\text{if}&j=\alpha,\:i\neq m;\\
s_m'-1&\text{if}&j=\alpha,\: i=m.\\
\end{array}\right.$$
Note that, since $\alpha(k_r+1)-1<n$ we have $a_{j,i}\geq0$ for all $j\in\{1,\ldots,\alpha\}$ and $i\in\{1,\ldots,r\}$. Moreover, if $m\neq r$ then
$$n-\displaystyle\sum_{j=1}^{\alpha}|I_j'^m| = n-(\alpha-2)(k_m+1)-k_m-|I_{\alpha}'^m|\geq n-(\alpha-1)(k_m+1)-(n-\alpha(k_m+1)-1) = k_m+2$$
and $n-\displaystyle\sum_{j=1}^{\alpha}|I_j'^r|=n-(\alpha-2)k_r-(k_r+1)-|I_{\alpha}'^r|\geq n-(\alpha-1)k_r-1-(n-\alpha k_r-1) = k_r$.
If $r=m$ we have
$$n-\displaystyle\sum_{j=1}^{\alpha}|I_j'^r|=n-(\alpha-2)k_r-(k_r+1)-|I_{\alpha}'^r|\geq n-(\alpha-1) k_r-1-(n-\alpha k_r-2)=k_r+1$$
Finally, for $i\neq m,r$ we have
$$n-\displaystyle\sum_{j=1}^{\alpha}|I_j'^i| = n-(\alpha-1)(k_i+1)-|I_{\alpha}'^i| \geq n-(\alpha-1)(k_i+1)-(n-\alpha(k_i+1)) = k_i+1$$
This yields that (i) in Lemma \ref{Lemma2} is satisfied by the sets $I_j'^i$. Moreover, (ii) is satisfied as well. Then, by Lemma \ref{Lemma2} there exists a rational map $\pi_{I_l'^m,I_{\alpha}'^m}$ making the following diagram commutative
\[
  \begin{tikzpicture}[xscale=3.9,yscale=-1.8]
    \node (A0_0) at (0, 0) {$\prod_{i=1}^r\G(k_i,n)$};
    \node (A0_1) at (1, 0) {$\P^{N'_{s,\ldots,s'}}$};
    \node (A1_1) at (1, 1) {$\prod_{i=1}^r\G(k_i,n-\sum_{j=1}^{\alpha} |I_j'^i|)$};
    \path (A0_0) edge [->,dashed]node [auto] {$\scriptstyle{\Pi_{T_{e_{I^1}},\ldots,e_{I^{\alpha-1}}}}$} (A0_1);
    \path (A0_1) edge [->,dashed]node [auto] {$\scriptstyle{\pi_{I_l'^m,I_{\alpha}'^m}}$} (A1_1);
    \path (A0_0) edge [->,dashed,swap]node [auto] {$\scriptstyle{\prod_{i=1}^r\Pi_{I'^i}}$} (A1_1);
  \end{tikzpicture}
\]
where $I'^i=\bigcup_{j=1}^{\alpha}I_j'^i$. 
Now, let $x=\Pi_{T_{e_{I_1},\ldots,e_{I_{\alpha-1}},e_{I_{\alpha}}}^{\hspace{0,1 cm}s\hspace{0,1 cm},\ldots,\hspace{0,1 cm}s,s'-1}}(\{V_i\}_{i=1}^r)$ be a general point in the image of $\Pi_{T_{e_{I_1},\ldots,e_{I_{\alpha-1}},e_{I_{\alpha}}}^{\hspace{0,1 cm}s\hspace{0,1 cm},\ldots,\hspace{0,1 cm}s,s'-1}}$, and $X\subset \prod_{i=1}^r\G(k_i,n)$ be the fiber of $\Pi_{T_{e_{I_1},\ldots,e_{I_{\alpha-1}},e_{I_{\alpha}}}^{\hspace{0,1 cm}s\hspace{0,1 cm},\ldots,\hspace{0,1 cm}s,s'-1}}$ over $x$. Set $x_{I_l'^m,I_{\alpha}'^m}=\pi_{I_l'^m,I_{\alpha}'^m}(x)$ and denote by $X_{I_l'^m,I_{\alpha}'^m}\subset \prod_{i=1}^r\G(k_i,n)$ the fiber of $\prod_{i=1}^r\Pi_{I'^i}$ over $x_{I_l'^m,I_{\alpha}'^m}$.
Therefore, $X\subset \displaystyle\bigcap_{I_l'^m,I_{\alpha}'^m} X_{I_l'^m,I_{\alpha}'^m}$, where the intersection runs over all pairs of sets $I_l'^m$ and $I_{\alpha}'^m$ with $|I_l'^m|=k_m$ and $|I_{\alpha}'^m|=s_m'-1$, and for all pairs of indexes $(l,m)\in\{1,\ldots,r\}\times\{1,\ldots,\alpha-1\}$.
In particular, if $\{W_i\}_{i=1}^r\in X$ is a general point then for every $m\in\{1,\dots,r\}$ we have $W_m\subset\displaystyle\bigcap_{I_l'^m,I_{\alpha}'^m}\langle e_i\:|\:i\in I'^m;V_m\rangle$, where the intersection runs over all pair of sets $I_l'^m$ and $I_{\alpha}'^m$ with $|I_l'^m|=k_m$ and $|I_{\alpha}'^m|=s_m'-1$, and $l\in\{1,\ldots,\alpha-1\}$.
Now, since $|I_l'^m|=k_m$, $s_m'-1\leq k_m$ and $l\in\{1,\ldots,\alpha-1\}$, we must have $\displaystyle\bigcap_{I_l'^m,I_{\alpha}'^m}\langle e_i\:|\:i\in I'^m\rangle=\emptyset$ and then $V_m=\displaystyle\bigcap_{I'^m}\langle e_i\:|\:i\in I'^m\:;\:V_m\rangle$ which yields $W_m=V_m$, for all $m=1,\ldots,r$. 

Now, assume that $n\geq k_r^2+k_r+1$. In this case we have that 
$$n-\alpha k_r -1\geq n-\dfrac{(n+1)}{k_r+1}k_r-1 = \dfrac{n(k_r+1)-(n+1)k_r-(k_r+1)}{k_r+1}\geq k_r$$
and for $i<r$ we have
$$n-\alpha(k_i+1)\geq n-\alpha(k_r)\geq n-\dfrac{(n+1)}{k_r+1}k_r=\dfrac{n-k_r}{k_r+1}\geq \dfrac{k_r^2+3k_r+1-k_r}{k_r+1}=k_r+1>k_i+1$$
Now, for each pair of indexes $(l,m)\in\{1,\ldots,\alpha\}\times\{1,\ldots,r\}$ we can consider subsets $I_j'^i\subseteq I_j^i$ with $|I_j'^i|=a_{j,i}$ for each $i\in\{1,\ldots,r\}$ and $j\in\{1,\ldots\alpha\}$ such that 
$$a_{j,i}=\left\{\begin{array}{lcl}
k_i&\text{if}& i=m,\: j=l\:\text{ or }\: i= r,\: j\neq l;\\
k_i+1&\text{if}& i=r\neq m,\: j=l\:\text{ or }\: i\neq m,r\:\text{ or }\:l\neq m.\\
\end{array}\right.$$
Therefore, arguing as in the proof of the first claim we conclude that  $\Pi_{T_{e_{I_1},\dots,e_{I_{\alpha}}}^{s,\ldots,s}}$ is birational.
\end{proof}

\subsection{Degenerating tangential projections to osculating projections}
In this section we study how the notion of osculating regularity introduced in \cite{MR19} behaves for products of Grassmannians. Let us recall \cite[Definition 5.5, Assumption 5.2]{MR19} and \cite[Definition 4.4]{AMR17}.

\begin{Definition}\label{osc_reg}
Let $X\subset\mathbb{P}^N$ be a projective variety. We say that $X$ has \textit{$m$-osculating regularity} if the following property holds: given general points $p_1,\dots,p_{m}\in X$ and an integer $s\geq 0$, 
there exists a smooth curve $C$ and morphisms $\gamma_j:C\to X$, $j=2,\dots,m$, 
such that  $\gamma_j(t_0)=p_1$, $\gamma_j(t_\infty)=p_j$, and the flat limit $T_0$ in the Grassmannian of the family of linear spaces 
$$
T_t=\left\langle T^{s}_{p_1},T^{s}_{\gamma_2(t)},\dots,T^{s}_{\gamma_{m}(t)}\right\rangle,\: t\in C\backslash \{t_0\}
$$
is contained in $T^{2s+1}_{p_1}$. We say that $\gamma_2,\dots, \gamma_m$ realize the $m$-osculating regularity of $X$ for $p_1,\dots,p_m.$

We say that $X$ has \textit{strong $2$-osculating regularity} if the following property holds: given general points $p,q\in X$ and  integers $s_1,s_2\geq 0$, there exists a smooth curve $\gamma:C\to X$ such that $\gamma(t_0)=p$, $\gamma(t_\infty)=q$ and the flat limit $T_0$ in the Grassmannian of the family of linear spaces 
$$
T_t=\left\langle T^{s_1}_p,T^{s_2}_{\gamma(t)}\right\rangle,\: t\in C\backslash \{t_0\}
$$
is contained in $T^{s_1+s_2+1}_p$.
\end{Definition}
For a discussion on the notions of $m$-osculating regularity and strong $2$-osculating regularity we refer to \cite[Section 5]{MR19} and \cite[Section 4]{AMR17}.

\begin{Proposition}\label{Strong_2_Osc_Reg}
The variety $\prod_{i=1}^r \G(k_i,n)$ has strong $2$-osculating regularity.
\end{Proposition}
\begin{proof}
Let $p,q\in\prod_{i=1}^r \G(k_i,n)$  be general points. We may assume that $p=e_{I_1}$ and $q=e_{I_2}$ with $I_1, I_2$ as in (\ref{Set_Indexes}) and consider the degree $r+\sum_{i=1}^r k_i$ rational normal curve given by
$$\gamma(s:t)=\prod_{i=1}^r (se_0+te_{k_r+1})\wedge\cdots\wedge(se_{k_i}+te_{k_r+k_i+1})$$
We work on the affine chart $s=1$ and set $t=(1:t)\in\p^1$. Now, consider the points
$$e_0,\ldots,e_n,e_0^t=e_0+te_{k_r+1},\ldots,e_{k_r}^t=e_{k_r}+te_{2k_r+1},e_{k_r+1}^t=e_{k_r+1},\ldots,e_{n}^t=e_{n}$$
and, for each $I=\{I^1,\ldots,I^r\}\in \Lambda$, the corresponding points in $e_{I}^t=e_{I^1}^t\otimes e_{I^2}^t\otimes\cdots\otimes e_{I^r}^t\in\prod_{i=1}^r \G(k_i,n)$ where, setting $I^j=\{i_1^j,\ldots,i_{k_j}^j\}$, $e_{I^j}^t=e_{i_1^j}^t\wedge\cdots\wedge e_{i_{k_j}^j}^t$.

Given integers $s_1,s_2\geq0$, let us consider the family of linear spaces
$$T_t=\langle T_{e_{I_1}}^{s_1},T_{e_{\gamma(t)}}^{s_2}\rangle, \:\:t\in\p^1\setminus\{0\}$$
By Proposition \ref{Osc_Product} we have
$$T_t=\langle e_{J}\: | \:d(I_1,J)\leq s_1\:;\:e_{J}^t\: | \:d(I_2,J)\leq s_2\rangle,\: t\neq 0$$
and
$$T_{e_{I_1}}^{s_1+s_2+1} = \langle e_{J}\: | \: d(I_1,J)\leq s_1+s_2+1 \rangle=\{Z_J=0\: | \:d(I_1,J)> s_1+s_2+1\}$$
Now, let $T_0$ be the flat limit of $\{T_t\}_{t\in\p^1\setminus \{0\}}$, we want to show that $T_0\subset T_p^{s_1+s_2+1}$. In order to do this it is enough to exhibit, for each index $I\in \Lambda$ with $d(I_1,I)> s_1+s_2+1$, a hyperplane $H_{I}$ of type $Z_{I}+t\left(\displaystyle\sum_{J\neq I}f_J(t)Z_J\right)=0$ where $f_J(t)\in\C[t]$ for every $J$. We define, for each $l\geq0$ and $I=\{I^1,\ldots,I^r\}\in \Lambda$,
$$\Delta(I,l)=\left\{\{(I^j\setminus J^j)\cup(J^j+k_r+1)\}_{1\leq j\leq r}\:|\:J^j\subset I^j\cap I_1^j\text{ and }\sum|J^j|=l\right\}\subset \Lambda$$
Furthermore, for each $l>0$ we define
$$\Delta(I,-l)=\{J\: | \:I\in \Delta(J,l)\};$$
$$s_{I}^+=\max_{l\geq0}\{\Delta(I,l)\neq\emptyset\}\in\{0,\ldots,\sum k_j+r\}$$
$$s_{I}^-=\max_{l>0}\{\Delta(I,-l)\neq\emptyset\}\in\{0,\ldots,\sum k_j+r\}$$
$$\Delta(I)^+=\displaystyle\bigcup_{0\leq l}\Delta(I,l)=\displaystyle\bigcup_{0\leq l\leq s_{I}^+}\Delta(I,l)$$
$$\Delta(I)^-=\displaystyle\bigcup_{0\leq l}\Delta(I,-l)=\displaystyle\bigcup_{0\leq l\leq s_I^-}\Delta(I,-l)$$
Now, let us write $e_I^t$ with $d(I_1,I)\leq s_2$, in the basis $e_J$ with $J\in\Lambda$. For any $I\in \Lambda$ we have
$$\begin{array}{ccl}
e_{I}^t&=&e_{I}+t\displaystyle\sum_{J\in\Delta(I,1)}\text{sign}(J)e_{J}+ \dots +t^{s_{I}^+}\displaystyle\sum_{J\in\Delta(I,s_{I}^+)}\text{sign}(J)e_{J}\\
&=&\displaystyle\sum_{l=0}^{s_{I}^+}\left(t^l\displaystyle\sum_{J\in\Delta(I,l)}\text{sign}(J)e_{J}\right)=\displaystyle\sum_{J\in\Delta(I)^+}t^{d(J,I)}\text{sign}(J)e_{J}.
\end{array}$$
where $\text{sign}(J)=\pm 1$. Note that $\text{sign}(J)$ depends on $J$ but not on $I$, then we can write $e_{I}^t=\displaystyle\sum_{J\in\Delta(I)^+}t^{d(J,I)}e_{J}$. Therefore, we have
$$T_t=\left\langle e_I\: | \:d(I_1,I)\leq s_1\: ; \:\displaystyle\sum_{J\in\Delta(I)^+}t^{d(J,I)}e_J\: | \:d(I_1,I)\leq s_2\right\rangle$$
Finally, we define
$$\Delta:=\{I\: : \:d(I_1,I)\leq s_1\}\bigcup\left(\displaystyle\bigcup_{ \:d(I_1,I)\leq s_2}\Delta(I)^+\right)\subset\Lambda$$
Let $I\in \Lambda$ be an index such that $d(I_1,I)>s_1+s_2+1$. If $I\notin \Delta$ then $T_t\subset\{Z_I=0\}$ for any $t\neq 0$ and we are done.

Assume that $I\in \Delta$. For any $e_K^t$ with non-zero coordinate $Z_I$ we have $I\in \Delta(K)^{+}$, that is $K\in \Delta(I)^-$. Now, it is enough to find a hyperplane $H_I$ of type 
$$
F_I=\displaystyle\sum_{J\in\Delta(I)^-}t^{d(J,I)}c_JZ_J=0
$$
with $c_J\in \C$ and $c_I\neq 0$, and such that $T_t\subset H_I$ for each $t\neq 0$. In the following, let us write $s_{i,I}^-=s$. Now, let us check what conditions we get by requiring $T_t\subset \{F_{I}=0\}$ for $t\neq 0$. Given $K\in \Delta(I)^-$ we have that $d(I,K)\leq s_{K}^+$ and 
$$\begin{array}{ccccl}
F_{I}(e_{K}^t)&=&F_{I}\left(\displaystyle\sum_{J\in\Delta(K)^+}t^{d(J,K)}e_{J}\right)&=&\displaystyle\sum_{J\in\Delta(I)^-}t^{d(J,I)}c_{J}\left(\displaystyle\sum_{J\in\Delta(K)^+}t^{d(J,K)}e_{J}\right)\\
&=&\displaystyle\sum_{J\in\Delta(I)^-\cap\Delta(K)^+}t^{d(J,I)+d(J,K)}c_{J}&=&t^{d(I,K)}\left[\displaystyle\sum_{J\in\Delta(I)^-\cap\Delta(K)^+}c_{J}\right]\\
\end{array}$$
Therefore, 
$$F_I(e_K^t)=0\:\:\forall\:t\neq0\Leftrightarrow\displaystyle\sum_{J\in\Delta(I)^-\cap\Delta(K)^+}c_{J}=0$$
Note that this is a linear condition on the coefficients $c_{J}$, with $J\in \Delta(I)^-$. Hence
\stepcounter{thm}
\begin{equation}\label{System1}
\begin{array}{ccl}
T_t\subset\{F_I=0\}\text{ for }t\neq0&\Leftrightarrow&\left\{\begin{array}{ll}F_{I}(e_{K})=0&\forall K\in \Delta(I)^-\cap B[I_1,s_1]\\
F_{I}(e_{K}^t)=0&\forall K\in \Delta(I)^-\cap B[I_1,s_2]\\
\end{array}\right.\\
&\Leftrightarrow&\left\{\begin{array}{ll}c_{K}=0&\forall K\in \Delta(I)^-\cap B[I_1,s_1]\\
\displaystyle\sum_{J\in\Delta(I)^-\cap\Delta(K)^+}c_J=0&\forall K\in \Delta(I)^-\cap B[I_1,s_2]\\
\end{array}\right.
\end{array}
\end{equation}
where $B[J,l]:=\{K\in\Lambda\:|\:d(J,K)\leq l\}$. The number of conditions on the $c_J$ is then $c:= |\Delta(I)^-\cap B[I_1,s_1]|+| \Delta(I)^-\cap B[I_1,s_2]|$.

The problem is now reduced to find a solution of the linear system given by the $c$ equations (\ref{System1}) in the $|\Delta(I)^-|$ variables $c_{J}$,  $J\in  \Delta(I)^-$ such that $c_I\neq 0$. Therefore, it is enough to find $s+1$ complex numbers $c_I=c_0\neq0,c_1,\ldots,c_s$ satisfying the following conditions
\stepcounter{thm}
\begin{equation}\label{system2}
\left\{\begin{array}{ll}c_j=0&\forall\:\: j=s,\ldots,d-s_1\\
\displaystyle\sum_{m=0}^{d(I,K)}|\Delta(I)^-\cap\Delta(K,l)|c_{d(I,K)-m}=0&\forall \:\:K\in \Delta(I)^-\cap B[I_1,s_2]\\
\end{array}\right.
\end{equation}
where $d=d(I_1,I)>s_1+s_2+1$. Note that (\ref{system2}) can be written as 
$$
\left\{\begin{array}{ll}c_j=0&\forall \:\:j=s,\ldots,d-s_1\\
\displaystyle\sum_{m=0}^j{{j}\choose{j-m}}c_m=0&\:\:\forall\:\: j=s,\ldots,d-s_2
\end{array}\right.
$$
that is
\stepcounter{thm}
\begin{equation}\label{system4}
\begin{array}{ll}
\left\{\begin{array}{l}
c_s=0\\
\vdots\\
c_{d-s_1}=0
\end{array}\right. & \left\{\begin{array}{l}
\binom{s}{0}c_{s}+\binom{s}{1}c_{s-1}+\dots+\binom{s}{s-1}c_1+\binom{s}{s}c_0=0\\
\vdots\\
\binom{d-s_2}{0}c_{d-s_2}+\binom{d-s_2}{1}c_{d-s_2-1}+\dots+\binom{d-s_2}{d-s_2-1}c_1+\binom{d-s_2}{d-s_2}c_0=0
\end{array}\right.
\end{array}
\end{equation}
Now, it is enough to show that the linear system (\ref{system4}) admits a solution  with $c_0\neq0$. If, $s<d-s_2$ then the system (\ref{system4}) reduces to $c_s=\cdots=c_{d-s_1}=0$ and then we can take $c_0=1$ and $c_1=\cdots=c_s=0$, since $d-s_1>s_2+1>1$.

So, let us assume that $s\geq d-s_2$. Since $c_s=\cdots=c_{d-s_1}=0$ our problem is translated into checking that the system (\ref{system4}) admits a solution involving the variables $c_{d-s_1-1},\dots,c_0$ with $c_0\neq 0$. First of all, note that the system (\ref{system4}) can be rewritten as follows
$$
\left\{\begin{array}{l}
\binom{s}{s-(d-s_1-1)} c_{d-s_1-1}+\binom{s}{s-(d-s_2-2)}c_{d-s_1-2}+\dots+\binom{s}{s-1}c_1+\binom{s}{s}c_0=0\\
\vdots\\
\binom{d-s_2}{d-s_2-(d-s_1-1)}c_{d-s_1-1}+\binom{d-s_2}{d-s_2-(d-s_1-2)}c_{d-s_1-2}+\dots+\binom{d-s_2}{d-s_2-1}c_1+\binom{d-s_2}{d-s_2}c_0=0
\end{array}\right.
$$
Thus, it is enough to check that the $(s-d+s_2+1)\times(d-s_1-1)$ matrix
$$
M=\left(\begin{matrix}
\binom{s}{s-(d-s_1-1)} & \binom{s}{s-(d-s_1-2)} & \dots & \binom{s}{s-1}\\
\vdots &\vdots & \ddots & \vdots\\
\binom{d-s_2}{d-s_2-(d-s_1-1)} &\binom{d-s_2}{d-s_2-(d-s_1-2)} & \dots & \binom{d-s_2}{d-s_2-1}
\end{matrix}\right)
$$
has maximal rank. Now, note that $s\leq d$ and $d>s_1+s_2+1$ yield $s-d+s_2+1<s-s_1\leq d-s_1$ and then $s-d+s_2+1\leq d-s_1-1$. Therefore, we have to show that the $(s-d+s_2+1)\times(s-d+s_2+1)$ submatrix
$$
\begin{array}{ccl}
M'&=& \left(\begin{matrix}
\binom{s}{s-d+s_2+1} & \binom{s}{s-d+s_2} & \dots & \binom{s}{1}\\
\vdots &\vdots & \ddots & \vdots\\
\binom{d-s_2}{s-d+s_2+1} &\binom{d-s_2}{s-d+s_2} & \dots & \binom{d-s_2}{1}
\end{matrix}\right)
\end{array}$$
has non-zero determinant. Finally, since $d-s_2>s_1+1\geq 1$ \cite[Corollary 2]{GV85} yields that $\det(M')\neq 0$.
\end{proof}

\begin{Proposition}\label{Osc_Reg}
Set $\alpha=\left\lfloor \frac{n+1}{k_r+1}\right\rfloor$. Then, the variety $\prod_{i=1}^r \G(k_i,n)$ has $\alpha$-osculating regularity.
\end{Proposition}
\begin{proof}
First of all, note that if $\alpha=2$ then the statement follows form Proposition \ref{Strong_2_Osc_Reg}. Then we may assume $\alpha\geq 3$.

Let $p_1,\ldots,p_{\alpha}\in\prod_{i=1}^r \G(k_i,n)$ be general points. We may assume that $p_j=e_{I_j}$ for $j=1,\ldots,\alpha$. Each $e_{I_j}$, $j\geq2$, is connected to $e_{I_1}$ by the degree $r+\sum_{i=1}^r k_i$ rational normal curve defined by
$$\gamma_j(s:t)=\prod_{i=1}^r(se_0+te_{(k_r+1)(j-1)})\wedge\cdots\wedge(se_{k_i}+te_{(k_r+1)(j-1)+k_i}) $$
We work on the affine chart $s=1$ and set $t=(1:t)$. Now, given $s\geq 0$ we consider the family of linear subspaces
$$T_t=\langle T_{e_{I_1}}^s,T_{\gamma_2(t)}^s,\ldots,T_{\gamma_{\alpha}(t)}^s\rangle,\:\:t\in \p^1\setminus\{0\}$$
Our goal is to show that the flat limit $T_0$ of $\{T_t\}_{t\in\p^1\setminus\{0\}}$ in $\G(\dim(T_t),N)$ is contained in $T^{2s+1}_{e_{I_1}}$. In order to do this, let us consider the points
$$e_0,\ldots,e_n,e_0^{j,t}=e_0+te_{(k_r+1)(j-1)},\ldots, e_{k_r}^{j,t}=e_{k_r}+te_{(k_r+1)j-1},e_{k_r+1}^{j,t}=e_{k_r+1},\ldots,e_n^{j,t}=e_n$$
and, for each $I=\{I^1,\ldots,I^r\}\in \Lambda$ and $j=2,\ldots,\alpha$, the corresponding points in $e_{I}^{j,t}=e_{I^1}^{j,t}\otimes e_{I^2}^{j,t}\otimes\cdots\otimes e_{I^r}^{j,t}\in\p^N$. By Proposition \ref{Osc_Product} we have
$$T_t=\langle e_{I}\: | \:d(I_1,I)\leq s;e_{I}^{j,t}\: | \:d(I_j,I)\leq s,\:j=2,\ldots,\alpha\rangle,\: t\neq 0$$
and
$$
T_{e_{I_1}}^{2s+1} = \langle e_{J}\: | \: d(I_1,J)\leq 2s+1 \rangle = \{Z_J=0\: | \:d(I_1,J)> 2s+1\}
$$
In order to show that $T_0\subset T_p^{2s+1}$, it is enough to exhibit, for each index $I\in \Lambda$ with $d(I_1,I)> 2s+1$, an hyperplane $H_{I}$ of type $Z_{I}+t\left(\displaystyle\sum_{J\neq I}f_J(t)Z_J\right)=0$ such that $T_t\subset \{H_i=0\}$ for $t\neq0$. 

For each $l\geq 0$, $j=2,\ldots,\alpha$ and $I=\{I^1,\ldots,I^r\}\in \Lambda$ we define
$$\Delta(I,l)_j=\left\{\{(I^k\setminus J^k)\cup(J^k+(j-1)(k_r+1)\}_{1\leq k\leq r}\:|\:J^k\subset I^k\cap I_1^k\text{ and }\sum|J^k|=l\right\}\subset \Lambda$$
where $L+\lambda=\{i+\lambda \:|\:i\in L\}$ is the translation of the set $L$ by the integer $\lambda$. For any $l>0$ we define
$$\Delta(I,-l)_j=\{J\: | \:I\in \Delta(J,l)_j\}$$
$$s_{I,j}^+=\max_{l\geq0}\{\Delta(I,l)_j\neq\emptyset\}\in\{0,\ldots,\sum k_j+r\}$$
$$s_{I,j}^-=\max_{l>0}\{\Delta(I,-l)_j\neq\emptyset\}\in\{0,\ldots,\sum k_j+r\}$$
$$\Delta(I)^+_j=\displaystyle\bigcup_{0\leq l}\Delta(I,l)_j=\displaystyle\bigcup_{0\leq l\leq s_{I,j}^+}\Delta(I,l)_j$$
$$\Delta(I)^-_j=\displaystyle\bigcup_{0\leq l}\Delta(I,-l)_j=\displaystyle\bigcup_{0\leq l\leq s_{I,j}^-}\Delta(I,-l)_j$$
Note that for any $l$ we have
\stepcounter{thm}
\begin{equation}\label{eq1}
J\in\Delta(I,l)_j\:\Rightarrow\:d(J,I)=|l|\:\text{ and }\:d(J,I_1)=d(I,I_1)+l
\end{equation}

We will write $e_I^t$ with $d(I_1,I)\leq s$, in the basis $e_J$ with $J\in\Lambda$. For any $I\in \Lambda$ we have
$$\begin{array}{ccl}
e_{I}^{j,t}&=&e_{I}\:+\:t\displaystyle\sum_{J\in\Delta(I,1)_j}\text{sign}(J)e_{J}\:+\: \cdots \: +\:t^{s_{I,j}^+}\displaystyle\sum_{J\in\Delta(I,s_{I,j}^+)}\text{sign}(J)e_{J}\\
&=&\displaystyle\sum_{l=0}^{s_{I,j}^+}\left(t^l\displaystyle\sum_{J\in\Delta(I,l)_j}\text{sign}(J)e_{J}\right)=\displaystyle\sum_{J\in\Delta(I)_j^+}t^{d(J,I)}\text{sign}(J)e_{J}.
\end{array}$$
where $\text{sign}(J)=\pm 1$. Note that $\text{sign}(J)$ depends on $J$ but not on $I$, then we can write $e_{I}^{j,t}=\displaystyle\sum_{J\in\Delta(I)_j^+}t^{d(J,I)}e_{J}$. Therefore, we have
$$T_t=\left\langle e_I\: | \:d(I_1,I)\leq s\: ; \:\displaystyle\sum_{J\in\Delta(I)_j^+}t^{d(J,I)}e_J\: | \:d(I_1,I)\leq s,\:2\leq j\leq\alpha\right\rangle$$
Finally we define
$$\Delta:=\{I\: : \:d(I_1,I)\leq s\}\bigcup\left(\displaystyle\bigcup_{\substack{\:d(I_1,I)\leq s\\2\leq j\leq\alpha}}\Delta(I)_j^+\right)\subset\Lambda$$
Let $I\in \Lambda$ be an index such that $d(I_1,I)>2s+1$. If $I\notin \Delta$, then $T_t\subset\{Z_I=0\}$ for any $t\neq 0$ and we are done. 

Now, assume that $I\in \Delta$. We will show that $\Delta(K_1)_{j_1}^+\cap\Delta(K_2)_{j_2}^+=\emptyset$ whenever $K_1,K_2\in \Lambda$ with $d(K_1,I_1),d(K_2,I_2)\leq s$ and $2\leq j_1,j_2\leq\alpha$ with $j_1\neq j_2$. 

In fact, suppose that $\Delta(K_1)_{j_1}^+\cap\Delta(K_2)_{j_2}^+\neq \emptyset$, that is there exists $I\in \Lambda$ such that 
$$I\in \Delta(K_1,l_1)_{j_1}\cap \Delta(K_2,l_2)_{j_2} \text{ for some } l_1 \text{ and } l_2$$
Now, consider the following sets
$$\begin{array}{ccl}
I^0&:=&I\cap I_1\\
I^1&:=&I\cap \{K_1+(j_1-1)(k_r+1)\}\\
I^2&:=&I\cap \{K_2+(j_2-1)(k_2+1)\}\\
I^3&:=&I\setminus (I^0\cup I^1\cup I^2)
\end{array}$$
Since $I\in \Delta(K_1,l_1)_{j_1}\cap \Delta(K_2,l_2)_{j_2}$ we have $|I^1|=l_1$ and $|I^2|=l_2$. Set $|I^3|=u$, then
$$d(I,I_1)=l_1+l_2+u\leq l_1+l_2+2u\stackrel{(\ref{eq1})}{=} d(K_1,I_1)+d(K_1,I_1)\leq 2s$$
contradicting $d(I_1,I)>2s+1$. Therefore we conclude that there is a unique $j_I$ for which 
$$I\in \bigcup_{d(I_1,I)\leq s}\Delta(I)_{j_I}^+$$
Now, let $J\in \Lambda$ such that $d(J,I_1)\leq s$ and $I\in \Delta(J)^+_{j_I}$. Note that 
$$d(I,I_1)-s(I)^-_{j_I}\leq d(I,I_1)-d(I,J)=d(J,I_1)\leq s\:\Rightarrow\:s+1-D+s(I)^-_{j_I}>0$$
where $D=d(I,I_1)>2s+1$. We define
$$\Gamma(I)=\displaystyle\sum_{0\leq l\leq s+1-D+s(I)^-_{j_I}}\Delta(I,-l)_{j_I}\subset\Gamma$$
Our aim now is to find a hyperplane of the form
\stepcounter{thm}
\begin{equation}\label{eqHI}
H_I=\left\{\displaystyle\sum_{J\in\Gamma(I)}t^{d(J,I)}c_JZ_J=0\right\}
\end{equation}
such that $T_t\subset H_I$ and $c_I\neq0$. First, note that
\stepcounter{thm}
\begin{equation}\label{eq3}
J\in \Gamma(I)\:\Rightarrow\:J\notin\displaystyle\bigcup_{\substack{\:d(I_1,K)\leq s\\2\leq j\leq\alpha\:;\:j\neq j_I}}\Delta(K)_j^+
\end{equation}
In fact, suppose that $J\in \Delta(I,-l)_{j_I}\cap\Delta(K,m)_j$, for some $K\in \Lambda$ with $d(K,I_1)\leq s$, and $0\leq j\leq s+1-D+s(I)^-_{j_I}$ with $j\neq j_I$. Then, since $J\in \Delta(I,-l)_{j_I}$ we have
$$|J\cap I_{j_I}|=|I\cap I_{j_I}|-l\geq s(I)^-_{j_I}-l\geq D-k-1>s$$
On the other hand, since $J\in \Delta(K,m)_j$ with $j\neq j_I$ we have
$$|J\cap I_{j_I}|=|K\cap I_{j_I}|\leq d(K,I_1)\leq s$$
which is a contradiction. Now, note that if $K\in \Lambda$ is such that $d(K,I_1)\leq s$ and $K\in \Gamma(I)$, then
$$d(K,I_1) = d(I,I_1)-d(I,K) > 2s+1-(s+1-D+s(I)^-_{j_I})> s+D-s(I)^-_{j_I}>s$$
Thus (\ref{eq3}) yields that the hyperplane $H_I$ given by (\ref{eqHI}) is such that 
$$\left\langle e_K\: | \:d(I_1,K)\leq s\: ; \:\displaystyle\sum_{J\in\Delta(K)_j^+}t^{d(J,K)}e_J\: | \:d(I_1,K)\leq s,\:2\leq j\leq \alpha\:;\:j\neq j_I\right\rangle\subset H_I,\:t\neq0$$
Therefore
$$T_t\subset H_I,\:t\neq0\:\Leftrightarrow\:\left\langle \displaystyle\sum_{J\in\Delta(K)_{j_I}^+}t^{d(J,K)}e_J\: | \:d(I_1,K)\leq s\right\rangle\subset H_I,\:t\neq0$$
Now, arguing as in the proof of Proposition \ref{Strong_2_Osc_Reg} we get
\stepcounter{thm}
\begin{equation}\label{eq4}
T_t\subset H_I,\:t\neq0\:\Leftrightarrow\: \displaystyle\sum_{J\in\Delta(K)_{j_I}^+\cap\Gamma(I)}c_J=0,\:\:\:\forall \:K\in\Delta(I)^-_{j_I}\cap B[I_1,s]
\end{equation}
So, the problem is reduced to find a solution $(c_J)_{J\in\Gamma(I)}$ for the linear system (\ref{eq4}) such that $c_I\neq 0$. We set $c_J=c_{d(I,J)}$ and reduce, as in the proof of Proposition \ref{Strong_2_Osc_Reg}, to the linear system
\stepcounter{thm}
\begin{equation}\label{eq5}
\displaystyle\sum_{l=0}^{s+1+D-s(I)^-_{j_I}}{{D-i}\choose{D-i-l}}c_l,\:\:D-S(I)^-_{j_I}\leq i\leq k
\end{equation}
We have $s+2+D-s(I)^-_{j_I}$ variables $c_0,\ldots,c_{s+1+D-s(I)^-_{j_I}}$ and $s+1+D-s(I)^-_{j_I}$ equations. Finally, the argument used in the last part of the proof of Proposition \ref{Strong_2_Osc_Reg} shows that the linear system (\ref{eq5}) admits a solution with $c_0\neq 0$.
\end{proof}

\section{On secant defectivity of products of Grassmannians}\label{sec2}
Let $X\subset\p^N$ be an irreducible non-degenerate variety of dimension $n$ and let
$$\Gamma_h(X)\subset X\times \dots \times X\times\G(h-1,N)$$
where $h\leq N$, be the closure of the graph of the rational map $\alpha: X\times  \dots \times X \dasharrow \G(h-1,N)$ taking $h$ general points to their linear span $\langle x_1, \dots , x_{h}\rangle$. Observe that $\Gamma_h(X)$ is irreducible and reduced of dimension $hn$. 

Let $\pi_2:\Gamma_h(X)\rightarrow\G(h-1,N)$ be the natural projection, and $\mathcal{S}_h(X):=\pi_2(\Gamma_h(X))\subset\G(h-1,N)$. Again $\mathcal{S}_h(X)$ is irreducible and reduced of dimension $hn$. Finally, consider
$$\mathcal{I}_h=\{(x,\Lambda) \: | \: x\in \Lambda\} \subset\p^N\times\G(h-1,N)$$
with natural projections $\pi_h$ and $\psi_h$ onto the factors. 

The {\it abstract $h$-secant variety} is the irreducible variety $\Sec_{h}(X):=(\psi_h)^{-1}(\mathcal{S}_h(X))\subset \mathcal{I}_h$. The {\it $h$-secant variety} is $\mathbb{S}ec_{h}(X):=\pi_h(Sec_{h}(X))\subset\p^N$. Then $\Sec_{h}(X)$ is an $(hn+h-1)$-dimensional variety.

The number $\delta_h(X) = \min\{hn+h-1,N\}-\dim\mathbb{S}ec_{h}(X)$ is called the \textit{$h$-secant defect} of $X$. We say that $X$ is \textit{$h$-defective} if $\delta_{h}(X) > 0$. We refer to \cite{Ru03} for a comprehensive survey on the subject.

Determining secant defectivity is a classical problem in algebraic geometry. A new strategy to determine the non secant defectivity was introduced in \cite[Theorem 5.3]{MR19}, the method is based on degenerating the span of several tangent spaces $T_{x_i}X$ in a single osculating space $T_x^sX$.

To state the criterion for non secant defectivity in \cite{MR19} we introduce a function $h_m:\mathbb{N}_{\geq0}\longrightarrow\mathbb{N}_{\geq0}$ counting how many tangent spaces can be degenerated into a higher order osculating space.

\begin{Definition}\label{h_m}
Given an integer $m\geq 0$ we define a function
$$h_m:\mathbb{N}_{\geq0}\longrightarrow\mathbb{N}_{\geq0}$$
as follows: for $h_m(0)=0$ and for any $k>0$ write
$$k+1=2^{\lambda_1}+2^{\lambda_2}+\cdots+2^{\lambda_l}+\varepsilon$$
where $\lambda_1>\lambda_2>\cdots>\lambda_l\geq1$ and $\varepsilon\in \{0,1\}$, then
$$h_m(k)=m^{\lambda_1-1}+m^{\lambda_2-1}+\cdots+m^{\lambda_l-1}$$
\end{Definition}

\begin{thm}\label{Theorem 5.3}\cite[Theorem 5.3]{MR19}
Let $X\subset \p^N$ be a projective variety having $m$-osculating regularity and strong $2$-osculating regularity. Let $s_1,\ldots,s_l\geq1$  integers such that the general osculating projection $\Pi_{p_1,\ldots,p_l}^{s_1,\ldots,s_l}$ is generically finite. If 
$$h\leq\displaystyle\sum_{j=1}^lh_m(s_j)$$
then $X$ is not $(h+1)$-defective.
\end{thm}

Now, we are ready to prove our main result on non-defectivity of product of Grassmannians. We follow the notation introduced in the previous sections.

\begin{thm}\label{Bound_Prod_Grass}
Assume that $n\geq 2k_r+1$. Set 
$$\alpha:=\left\lfloor\dfrac{n+1}{k_r+1}\right\rfloor$$
and let $h_{\alpha}$ be as in Definition \ref{h_m}. Assume that
\begin{itemize}
\item[-] either $n\geq k_r^2+3k_r+1$ and $h\leq \alpha h_{\alpha}(\sum_{i=1}^r k_i+r-2)$ or
\item[-] $\alpha(k_r+1)-1<n<k_r^2+3k_r+1$ and $h\leq (\alpha-1) h_{\alpha}(\sum_{i=1}^r k_i+r-2)+h_{\alpha}(s')$
\end{itemize}
where $s'=\sum_{i=1}^r s_i-2$ with $s_i'=\min\{k_i+1,n-\alpha(k_i+1)\}$ for $i\neq r$ and $s_r'=\min\{k_r,n-\alpha k_r-1\}$. Then $\prod_{i=1}^r\G(k_i,n)$ is not $(h+1)$-defective. 
\end{thm}
\begin{proof}
We have shown in Propositions \ref{Osc_Reg}, \ref{Strong_2_Osc_Reg} that $\prod_{i=1}^r\G(k_i,n)$ has respectively $\alpha$-osculating regularity for $\alpha:=\left\lfloor\frac{n+1}{k_r+1}\right\rfloor$, and strong $2$-osculating regularity. The statement then follows immediately from Proposition \ref{Osc_Proj_Product} and Theorem \ref{Theorem 5.3}.
\end{proof}

\begin{Corollary}\label{main_cor}
The variety $\prod_{i=1}^r\G(k_i,n)$ is not $(h+1)$-defective for $h\leq \left(\dfrac{n+1}{k_r+1}\right)^{\lfloor \log_2(\sum k_j+r-1)\rfloor}$.
\end{Corollary}
\begin{proof}
We may write
\stepcounter{thm}
\begin{equation}\label{h_m(s)}
\sum_{i=1}^r k_i+r-1=2^{\lambda_1}+2^{\lambda_2}+\cdots+2^{\lambda_l}+\varepsilon
\end{equation}
with $\lambda_1>\lambda_2>\cdots>\lambda_l\geq1$ and $\varepsilon\in \{0,1\}$. Then $h_{\alpha}(\sum_{i=1}^r k_i+r-2)=\alpha^{\lambda_1-1}+\alpha^{\lambda_2-1}+\cdots+\alpha^{\lambda_l-1}$.

The first bound in Theorem \ref{Bound_Prod_Grass} gives $h\leq \alpha^{\lambda_1}+\cdots+\alpha^{\lambda_l}$. Furthermore, considering just the first summand in the second bound in Theorem \ref{Bound_Prod_Grass} we get that $\prod_{i=1}^r\G(k_i,n)$ is not $(h+1)$-defective for $h\leq (\alpha-1)(\alpha^{\lambda_1-1}+\alpha^{\lambda_2-1}+\cdots+\alpha^{\lambda_l-1})$.

Finally, from (\ref{h_m(s)}) we get that $\lambda_1=\lfloor \log_2(r-1+\sum k_i)\rfloor$. Hence, asymptotically we have $h_{\alpha}(\sum k_j+r-2)\sim\alpha^{\lfloor \log_2(r-1+\sum k_i)\rfloor-1}$, and by Theorem \ref{Bound_Prod_Grass} if $h\leq  \left(\dfrac{n+1}{k_r+1}\right)^{\lfloor \log_2(\sum k_j+r-1)\rfloor}$ then the variety $\prod_{i=1}^r\G(k_i,n)$ is not $(h+1)$-defective.
\end{proof}

\section{On secant defectivity of flag varieties}\label{sec3}
Our goal is to compute the higher osculating spaces of $\F(k_1,\ldots,k_r;n)$. In order to do this, we will use the following notion introduced in \cite[Definition 3.2]{FMR18}.

\begin{Definition}\label{Wellbehaved}
Let $X\subset\mathbb{P}^N$ be an irreducible variety and $Y=\mathbb{P}^k\cap X$ be a linear section of $X$. We say that $Y$ is \textit{osculating well-behaved} if for each smooth point $p\in Y$ we have 
$$T_p^sY=\mathbb{P}^k\cap T_p^sX$$ 
for every $s\geq 0$. 
\end{Definition}

Let us denote by $M_i$ the following $(k_i+1)\times(n+1)$ matrix
$$
M_i=\left[\begin{array}{cccc}
\textit{I}_{k_1+1}& \dots & \dots &(x_{l,m}^1)_{\substack{0\leq l\leq k_1\\k_1+1\leq m\leq n}}\\
0&\textit{I}_{k_2-k_1}& \dots &(x_{l,m}^2)_{\substack{k_1+1\leq l\leq k_2\\k_2+1\leq m\leq n}}\\
\vdots& \ddots &\ddots& \vdots\\
0&\cdots&I_{k_i-k_{i-1}}&(x_{l,m}^i)_{\substack{k_{i-1}+1\leq l\leq k_i\\k_i+1\leq m\leq n}}\\
\end{array}\right]
$$
and consider the map
$$
\begin{array}{cccl}
\varphi':&\prod_{i=1}^r\C^{(k_1+1)(n-k_1)+\sum_{j=2}^i(n-k_j)(k_j-k_{j-1})}&\longrightarrow&\p^N\\
&(M_1,\ldots,M_r)&\longmapsto&\left(\prod_{i=1}^r\det(M_{J^i})\right)_{J=\{J^1,\ldots,J^r\}\in \Lambda}
\end{array}
$$
where $M_{J^i}$ is the submatrix obtained from $M_i$ by considering only the columns indexed by $J^i$. 

For each $2\leq i\leq r$ and $m\leq k_l$, let us take $x^i_{l,m}=0$ in $M_i$. Then $\varphi'$ becomes the parametrization $\varphi$ of $\prod_{i=1}^r\displaystyle\G(k_i,n)$ in (\ref{Param_Product}).

Now, set $x_{l,m}^i=x_{l,m}^r$ in $M_i$ for each $i=1,\ldots,r-1$ and $1\leq  l<m\leq n$. Hence $\varphi$ becomes the parametrization of $\F(k_1,\ldots,k_r;n)$ given by 
$$\begin{array}{cccl}
\overline{\varphi}:&\C^{(k_1+1)(n-k_1)+\sum_{j=2}^r(n-k_j)(k_j-k_{j-1})}&\longrightarrow&\mathbb{P}(\Gamma_a)\subset\P^N\\
&M_r&\longmapsto&\varphi\left(\overline{M}_1,\ldots,\overline{M}_r\right)
\end{array}$$
where $\overline{M}_i$ is the submatrix obtained from $M_r$ by considering only the first $k_i+1$ rows.

\begin{Lemma}
Let $T_{\varphi'}^{s}\left(\prod_{i=1}^r\G(k_i,n)\right):=\left\langle\dfrac{\partial^{|I|}\varphi'}{\partial x_{|I|}}(0)\:|\:|I|\leq s\right\rangle$ be the $s$-osculating space of $\prod_{i=1}^r\G(k_i,n)$ with respect to $\varphi'$. Then $T_{\varphi'}^{s}\left(\prod_{i=1}^r\G(k_i,n)\right)=T^{s}\left(\prod_{i=1}^r\G(k_i,n)\right)$ for every $s\leq r+\sum_{i=1}^rk_{i}$. In particular, 
$$
\frac{\partial^{s}\varphi'}{\partial x_{|I|}}(0)=\frac{\partial^{|J|}\varphi}{\partial x_{|J|}}(0)
$$
for some $J$ with $|J|\leq|I|$.
\end{Lemma}
\begin{proof}
First, note that if for any $x_{l,m}^i\in x_{|I|}$ we have $m>k_i$, then $\frac{\partial^{s}\varphi'}{\partial x_{|I|}}(0)=\frac{\partial^{|I|}\varphi}{\partial x_{|I|}}(0)$ and we are done.

Now, let $2\leq i\leq r$ and consider a derivative $\frac{\partial^{|I|}\varphi'}{\partial x_{|I|}}(0)$ 
such that $x_{l,m}^{i}\in x_{|I|}$ with $m\leq k_{i}$. Therefore, to prove the statement it is enough to show that this partial derivative can be written in terms of another partial derivative $\dfrac{\partial^{|J|}\varphi'}{\partial x_{|J|}}(0)$ with $x_{l,m}^{i}\notin x_{|J|}$, $m\leq k_{i}$ and $|J|<|I|$.

Fix $2\leq i\leq r$ and let $x_{l_{1},m_{1}}^{i},...,x_{l_{h},m_{h}}^{i},x_{l_{h+1},m_{h+1}}^{i},...,x_{l_{b},m_{b}}^{i}\in x_{|I|}$ such that $m_{a}\leq k_{i}$ for every $a=1,...,h$ and $b\leq k_{i}+1$.

If $\frac{\partial^{b}\varphi'}{\partial x_{l_{1},m_{1}}^{i}\cdots\partial x_{l_{h},m_{h}}^{i}\partial x_{l_{h+1},m_{h+1}}^{i}\cdots\partial x_{l_{b},m_{b}}^{i}}(0)\neq 0$ consider the minor $M_{J^i}$ of $M_i$ such that the monomial 
$$x_{l_{1},m_{1}}^{i}\dots x_{l_{h},m_{h}}^{i} x_{l_{h+1},m_{h+1}}^{i}\dots x_{l_{b},m_{b}}^{i}$$
appears in the expression of $\det(M_{J^i})$. Then, there exist variables $x_{\sigma_{J^i}(l_{h+1}),\sigma_{J^i}(m_{h+1})}^{i},...,x_{\sigma_{J^i}(l_{b}),\sigma_{J^i}(m_{b})}^{i}$ such that $x_{\sigma_{J^i}(l_{h+1}),\sigma_{J^i}(m_{h+1})}^{i}\cdots x_{\sigma_{J^i}(l_{b}),\sigma_{J^i}(m_{b})}^{i}$ is also a monomial in $\det(M_{J})$, where $\sigma_{J^i}$ is a permutation on the indexes such that $\sigma_{J^i}(m_a)>k_{i}$ for all $h+1\leq a\leq b$.

This shows that

$$\frac{\partial^{m}\varphi'}{\partial x_{l_{1},m_{1}}^{i}\cdots\partial x_{l_{h},m_{h}}^{i}\partial x_{l_{h+1},m_{h+1}}^{i}\cdots\partial x_{l_{b},m_{b}}^{i}}(0)=\frac{\partial^{m}\varphi'}{\partial x_{\sigma_{J^i}(l_{h+1}),\sigma_{J^i}(m_{h+1})}^{i},...,\partial x_{\sigma_{J^i}(l_{b}),\sigma_{J^i}(m_{b})}^{i}}(0)$$

We have thus decreased the number of variables with respect we differentiate
and thus lowered the order of the derivatives. Finally, since $\frac{\partial\varphi}{\partial x_{l,m}^{i}}(0)=\frac{\partial\varphi'}{\partial x_{l,m}^{i}}(0)$ for $m>k_{i}$ we are done.
\end{proof}

\begin{Lemma}\label{Prop2}
Since $\overline{\varphi}$ is a sub-parametrization of $\varphi'$ by the chain rule we have
$$
\dfrac{\partial^{s}\overline{\varphi}}{\partial x_{|I|}}(0)=\underset{|K|}{\sum}\dfrac{\partial^{s}\varphi'}{\partial x_{|K|}}(0){=}\underset{|J|}{\sum}\dfrac{\partial^{s}\varphi}{\partial x_{|J|}}(0)
$$
where $|K|=|I|=s$ and $|J|\leq |I|$. Let $\dfrac{\partial^{s}\overline{\varphi}}{\partial x_{|I|}}(0)\neq0$ with $|I|=s$ such that for each $x_{l,m}^i\in x_{|I|}$ we have that $m>k_i$. Then, in the above decomposition there is at least a vector $\dfrac{\partial^{s}\varphi}{\partial x_{|J|}}(0)$ with $|J|=s$.
\end{Lemma}
\begin{proof}
For any $x_{l,m}^{i}\in x_{|I|}$ let $h(m)$ be the maximum index in $\{1,...,r\}$ such that $m>k_{h(m)}$. Since for each $x_{l,m}^i\in x_{|I|}$ we have that $m>k_i$ and $\frac{\partial^{s}\overline{\varphi}}{\partial x_{|I|}}(0)\neq0$, we get that any $x_{l,m}^{i}\in x_{|I|}$ appears at most $h(m)$ times in $x_{|I|}$.

Now, for any $s\leq h(m)$, the chain rule expression of $\frac{\partial^{s}\overline{\varphi}}{(\partial x_{l,m}^{i})^{s}}(0)$ contains the factor
 
$$\dfrac{\partial^{s}\varphi'}{\partial x_{l,m}^{1}\partial x_{l,m}^{2}...\partial x_{l,m}^{h(m)}}(0)=\dfrac{\partial^{s}\varphi}{\partial x_{l,m}^{1}\partial x_{l,m}^{2}...\partial x_{l,m}^{h(m)}}(0)$$
 
Repeating this argument for all indexes $x_{l,m}^{i}\in x_{|I|}$ we conclude.
\end{proof}

\begin{Proposition}\label{well-behaivior}
The flag variety is osculating well-behaved, that is
$$
T_p^s\F(k_1,\ldots,k_r;n)=T_p^s\prod_{i=1}^r\G(k_i,n)\cap \mathbb{P}(\Gamma_a)
$$
for any $p \in \F(k_1,\ldots,k_r;n)$ and non-negative integer $s$.
\end{Proposition}
\begin{proof}
We may assume that $p=e_{I}$ where $I=\{I^1,\ldots,I^r\}$ and $I^l=\{0,\ldots,k_l\}$ for each $1\leq l\leq r$. Let us first assume that $s=r+\sum _{i=1}^r k_i$. Note that $s$ is the smallest integer for which $T_p^s\F(k_1,\ldots,k_r;n)=\mathbb{P}(\Gamma_a)$ and $T_p^s\prod_{i=1}^s  \G(k_i,n)=\p^N$, in this case $T_p^s\F(k_1,\ldots,k_r;n)=\mathbb{P}(\Gamma_a)=\mathbb{P}(\Gamma_a)\cap\p^N=\mathbb{P}(\Gamma_a)\cap T_p^s\prod_{i=1}^s\G(k_i,n)$
and we are done. Now, assume $s<r+\sum _{i=1}^r k_i$. Let
\stepcounter{thm}
\begin{equation}\label{v_in_s-1}
v=\displaystyle\sum_{|I|\leq s-1}\alpha_{|I|}\dfrac{\partial^{|I|}\varphi}{\partial x_{|I|}}(0)
\end{equation}
be a general vector in $T_p^{s-1}\prod_{i=1}^{r}  \G(k_i,n)$, and assume that 
\begin{center}
$v\in T_p^{s-1}\prod_{i=1}^r  \G(k_i,n)\cap \mathbb{P}(\Gamma_a)\subset T_p^{s}\prod_{i=1}^r  \G(k_i,n)\cap \mathbb{P}(\Gamma_a)= T_p^s\F(k_1,\ldots,k_r;n)$
\end{center}
this yields that $v$ can be written as 
\stepcounter{thm}
\begin{equation}\label{v_in_s}
v=\displaystyle\sum_{|I|\leq s-1}\beta_{|I|}\dfrac{\partial^{|I|}\overline{\varphi}}{\partial^{|I|}x_{|I|}}(0)+\displaystyle\sum_{|I|=s}\beta_{|I|}\dfrac{\partial^{|I|}\overline{\varphi}}{\partial^{|I|}x_{|I|}}(0)
\end{equation}
Now, recall that for any $I$ such that there are variables $x_{l,m}^i\in x_{|I|}$ with $m\leq k_i$ we can find another set $J$ for which $|J|<|I|$ and 
$$\dfrac{\partial^{s}\varphi'}{\partial x_{|I|}}(0)=\dfrac{\partial^{|J|}\varphi}{\partial x_{|J|}}(0)$$
Therefore, we can assume that any set $I$ in the second summand of (\ref{v_in_s}) is such that $m>k_i$ for any $x_{l,m}^i\in x_{|I|}$. Thus, by Lemma \ref{Prop2}, we will have an equality in (\ref{v_in_s-1}) and (\ref{v_in_s}) if and only if $\beta_{|I|}=0$ for any set $I$ such that $m>k_i$ for all $x_{l,m}^i\in x_{|I|}$. Hence $v\in T_p^{s-1}\F(k_1,\ldots,k_r;n)$.
\end{proof}

\subsection{Osculating Projections}
Let $s_1,\ldots,s_{\alpha}$ be integers such that $0\leq s_m\leq r-2+\sum_{i=1}^r k_i$. Denote $T_p^s\F(k_1,\ldots,k_r;n)$ simply by $T_p^s\F$ and the linear subspace $\langle T_{e_{I_1}}^{s_1}\F,\ldots,T_{e_{I_m}}^{s_m}\F\rangle$ by $T_{e_{I_1},\ldots,e_{I_m}}^{s_1,\ldots,s_m}\F$. Then, for $m\leq \alpha$ we have the linear projection 
$$\Pi_{T_{e_{I_1},\ldots,e_{I_{m}}}^{\hspace{0,1 cm}s_1\hspace{0,1 cm},\ldots,\hspace{0,1 cm}s_{m}}\F}:\F(k_1,\ldots,k_r;n)\dashrightarrow\p^{N_{s_1,\ldots,s_m}}$$ 
\begin{Proposition}\label{Osc_Proj_Flag}
Let $I_1,\ldots,I_{\alpha}$ be as in (\ref{Set_Indexes}) and $s=r-2+\sum_{i=1}^r k_i$. Then, 
\begin{itemize}
\item[-]$\Pi_{T_{e_{I_1},\ldots,e_{I_{\alpha-1}},}^{s,\ldots,s}\F}$ is birational;
\item[-]$\Pi_{T_{e_{I_1},\ldots,e_{I_{\alpha}}}^{s,\ldots,s}\F}$ is birational whenever $n\geq k_r^2+3k_r+1$. 
\end{itemize}
\end{Proposition}
\begin{proof}
Since $\Pi_{T_{e_{I_1},\ldots,e_{I_{\alpha-1}}}^{s,\ldots,s}}$ factors trough $\Pi_{T_{e_{I_1},\ldots,e_{I_{\alpha-1}}}^{s,\ldots,s}\F}$, it is enough to show that the restriction of $\Pi_{T_{e_{I_1},\ldots,e_{I_{\alpha-1}}}^{s,\ldots,s}}$ to $\F(k_1,\ldots,k_r)$ is birational.

For any $i\neq r$ and  $1\leq j\leq\alpha-1$ consider $I_j'^i= I_j^i$ and $I_j'^r\subset I_j^r$ of cardinality $k_r$. Since $n\geq 2k_r+1$ and $k_r\geq k_i$ we must have
$$
\begin{array}{ccl}
 n-\displaystyle\sum_{j=1}^{\alpha-1}|I_j'^i|&=&n-(\alpha-1)(k_i+1)\geq n-(\alpha-1)k_r\geq k_r+1\leq k_i+1
\end{array}$$
Now, let us denote by $I'^i$ the union $\bigcup_{j=1}^{\alpha-1}I_j'^i$. Then, by Lemma \ref{Lemma2} there exists a rational map $\pi_{I'^r}$ making the following diagram commutative
\[
\begin{tikzpicture}[xscale=3.9,yscale=-1.8]
    \node (A0_0) at (0, 0) {$\F(k_1,\ldots,k_r;n)$};
    \node (A0_1) at (1, 0) {$\P^{N'_{s,\ldots,s}}$};
    \node (A1_1) at (1, 1) {$\prod_{i=1}^r\G(k_i,n-\sum_{j=1}^{\alpha} |I_j'^i|)$};
    \path (A0_0) edge [->,dashed]node [auto] {$\scriptstyle{\Pi_{T_{e_{I_1},\ldots,e_{I_{\alpha-1}},e_{I_{\alpha}}}^{\hspace{0,1 cm}s\hspace{0,1 cm},\ldots,\hspace{0,1 cm}s,s'}}}$} (A0_1);
    \path (A0_1) edge [->,dashed]node [auto] {$\scriptstyle{\pi_{I'^r}}$} (A1_1);
    \path (A0_0) edge [->,dashed,swap]node [auto] {$\scriptstyle{\prod_{i=1}^r\Pi_{I'^i}}$} (A1_1);
\end{tikzpicture}
\]
Now, let $x=(\{V_i\}_{i=1}^r)$ be a general point in the image of $\Pi_{T_{e_{I_1},\ldots,e_{I_{\alpha-1}}}^{s,\ldots,s}}$ and $X\subset \F(k_1,\ldots,k_r;n)$ be the fiber of $\Pi_{T_{e_{I_1},\ldots,e_{I_{\alpha-1}}}^{s,\ldots,s}}$ over $x$. Set $x_{I'^r}=\pi_{I'^r}(x)$ and denote by $X_{I'^r}\subset \F(k_1,\ldots,k_r;n)$ the fiber of $\prod_{i=1}^r\Pi_{I'^i}$ over $x_{I'^r}$.

Therefore, $X\subset \displaystyle\bigcap_{I'^r} X_{I'^r}$, where the intersection runs over all sets $I'^r=\bigcup_{j=1}^{\alpha-1}I_j'^r$ with $I_j'^r\subset I_j^r$ and $|I_j'^r|=k_r$ for $1\leq j \leq\alpha-1$.

Now, note that if $\{W_i\}_{i=1}^r\in X$ is a general point, then we must have $W_i\subset\langle e_m\:|\:m\in\bigcup_{j=1}^{\alpha} I_j'^i\:;\:V_i\rangle$ for any choice of $\bigcup_{j=1}^{\alpha} I_j'^i$. Hence,
\stepcounter{thm}
\begin{equation}\label{W_i}
W_i\subset\displaystyle\bigcap_{I'^r}\langle e_m\:|\:m\in\bigcup_{j=1}^{\alpha} I_j'^i\:;\:V_i\rangle
\end{equation}
In particular, $W_r\subset\displaystyle\bigcap_{I'^r}\langle e_m\:|\:m\in\bigcup_{j=1}^{\alpha} I_j'^r;V_r\rangle$. Now, since $|I_j'^r|\leq k_r$ we must have $\displaystyle\bigcap_{I'^r}\langle e_m\:|\:m\in\bigcup_{j=1}^{\alpha} I_j'^r\rangle=\emptyset$ and then $V_r=\displaystyle\bigcap_{I'^r}\langle e_m\:|\:m\in\bigcup_{j=1}^{\alpha} I_j'^r\:;\:V_r\rangle$ which yields $W_r=V_r$.

Now, set $i\leq r-1$. Since $\{V_i\}_{i\in K}$ is general in $\F(k_1,\ldots,k_r;n)$ and $n-\displaystyle\sum_{j=1}^{\alpha-1}|I_j'^i|\geq k_r+1$ we have $V_r\cap\langle e_m\:|\:m\in\bigcup_{j=1}^{\alpha} I_j^i\rangle=\emptyset$. On the other hand $W_i\subset W_r=V_r$ for all $i\leq r-1$, then $W_i\cap\langle e_m\:|\:m\in\bigcup_{j=1}^{\alpha} I_j^i\rangle=\emptyset$. Hence, by (\ref{W_i}) we must have $W_i=V_i$ for any $i\leq r-1$.

Now, assume that $n\geq k_r^2+3k_r+1$ then 
$$
\begin{array}{ccl}
n-\alpha(k_i+1)\geq n-\alpha k_r&\geq& n-\dfrac{(n+1)}{k_r+1}k_r = \dfrac{n(k_r+1)-(n+1)k_r}{k_r+1}\\
&=& \dfrac{n-k_r}{k_r+1}\geq \dfrac{k_r^2+3k_r+1-k_r}{k_r+1} = k_r+1
\end{array} 
$$

Then, arguing as in the proof of the first case, for any choice of subsets $I_j'^i\subset I_j^i$, $I_j'^i= I_j^i$ with $i\neq r$ and $1\leq j\leq\alpha-1$, $I_j'^r\subsetneq I_j^r$ of cardinality $k_r$ we get, by Lemma \ref{Lemma2}, a rational map $\pi_{I'^r}$ making the following diagram commutative
\[
\begin{tikzpicture}[xscale=3.9,yscale=-1.8]
    \node (A0_0) at (0, 0) {$\F(k_1,\ldots,k_r;n)$};
    \node (A0_1) at (1, 0) {$\P^{N'_{s,\ldots,s}}$};
    \node (A1_1) at (1, 1) {$\prod_{i=1}^r\G(k_i,n-\sum_{j=1}^{\alpha} |I_j'^i|)
$};
    \path (A0_0) edge [->,dashed]node [auto] {$\scriptstyle{\Pi_{T_{e_{I_1}},\ldots,e_{I_{\alpha}}}}$} (A0_1);
    \path (A0_1) edge [->,dashed]node [auto] {$\scriptstyle{\pi_{I'^r}}$} (A1_1);
    \path (A0_0) edge [->,dashed,swap]node [auto] {$\scriptstyle{\prod_{i=1}^r\Pi_{I'^i}}$} (A1_1);
\end{tikzpicture}
\]
where $I'^i=\bigcup_{j=1}^{\alpha}I_j'^i$, $i=1,\ldots,r$. Now, to conclude it is enough to follow the same argument used in the end of the proof of the first claim.
\end{proof}

\subsection{Non-Secant defectivity of flag varieties}
We recall \cite[Proposition 4.4]{FMR18} which describes how the notion of osculating regularity behaves under linear sections.
\begin{Proposition}\label{Prop5}
Let $X\subset\mathbb{P}^N$ be an irreducible projective variety and $Y=\mathbb{P}^k\cap X$ a linear section of $X$ that is osculating well-behaved. Assume that given general points $p_1,\ldots,p_m\in Y$ one can find smooth curves $\gamma_j:C\rightarrow X, j=2,\dots,m,$ realizing the $m$-osculating regularity of $X$ for $p_1,\ldots,p_m$ such that $\gamma_j(C)\subset Y.$  Then $Y$ has $m$-osculating regularity as well. Furthermore, the analogous statement for strong $2$-osculating regularity holds as well.
\end{Proposition}

\begin{Proposition}
The flag variety $\F(k_1,\ldots,k_r;n)$ has strong $2$-osculating regularity and $\alpha$-osculating regularity, where $\alpha:=\left\lfloor\frac{n+1}{k_r+1}\right\rfloor$.
\end{Proposition}
\begin{proof}
The statement follows immediately from Propositions \ref{Strong_2_Osc_Reg}, \ref{Osc_Reg}, \ref{Prop5}.
\end{proof}

Now, we are ready to prove our main result on non-defectivity of flags varieties.

\begin{thm}\label{Bound_Flags}
Assume that $n\geq 2k_r+1$. Set 
$$\alpha:=\left\lfloor\dfrac{n+1}{k_r+1}\right\rfloor$$
and let $h_{\alpha}$ be as in Definition \ref{h_m}. If either
\begin{itemize}
\item[-] $n\geq k_r^2+3k_r+1$ and $h\leq \alpha h_{\alpha}(\sum k_j+r-2)$ or
\item[-] $n< k_r^2+3k_r+1$ and $h\leq (\alpha-1) h_{\alpha}(\sum k_j+r-2)$.
\end{itemize}
Then, $\F(k_1,\ldots,k_r;n)$ is not $(h+1)$-defective. In particular, if
$$h\leq \left(\dfrac{n+1}{k_r+1}\right)^{\lfloor \log_2(\sum k_j+r-1)\rfloor}$$
then $\F(k_1,\ldots,k_r;n)$ is not $(h+1)$-defective.
\end{thm}
\begin{proof}
The first part is an immediately consequence of Propositions \ref{Prop5}, \ref{Osc_Proj_Flag} and Theorem \ref{Theorem 5.3}. For the last claim note that if we write
\stepcounter{thm}
\begin{equation}\label{h_m(s)2}
\sum k_j+r-1=2^{\lambda_1}+2^{\lambda_2}+\cdots+2^{\lambda_l}+\varepsilon
\end{equation}
with $\lambda_1>\lambda_2>\cdots>\lambda_l\geq1$ and $\varepsilon\in \{0,1\}$. Then
$$h_{\alpha}(\sum k_j+r-2)=\alpha^{\lambda_1-1}+\alpha^{\lambda_2-1}+\cdots+\alpha^{\lambda_l-1}$$
Therefore, the first bound in Theorem \ref{Bound_Flags} yields
$$h\leq \alpha^{\lambda_1}+\alpha^{\lambda_2}+\cdots+\alpha^{\lambda_l}$$

Furthermore, by the second bound in Theorem \ref{Bound_Flags} we get that $\F(k_1,\ldots,k_r;n)$ is not $(h+1)$-defective for
$$h\leq (\alpha-1)(\alpha^{\lambda_1-1}+\alpha^{\lambda_2-1}+\cdots+\alpha^{\lambda_l-1})$$

Finally, by (\ref{h_m(s)2}) we get that $\lambda_1=\lfloor \log_2(\sum k_j+r-1)\rfloor$. Hence, asymptotically we have $h_{\alpha}(\sum k_j+r-2)\sim\alpha^{\lfloor \log_2(\sum k_j+r-1)\rfloor}$, and by Theorem \ref{Bound_Flags} for $h\leq \alpha^{\lfloor \log_2(\sum k_j+r-1)\rfloor}\leq \left(\dfrac{n+1}{k_r+1}\right)^{\lfloor \log_2(\sum k_j+r-1)\rfloor}$ the flag variety $\F(k_1,\ldots,k_r;n)$ is not $(h+1)$-defective.
\end{proof}

\begin{Remark}\label{reduction}
Now, given a flag $\F(k_1,\ldots,k_r;n)$ with $n<2k_r+1$. Assume that $n\geq 2k_{j}+1$ for some index $j$ and let $l$ be the maximum among these j's. Then we have a natural projection
$$\begin{array}{ccccl}
\pi&:&\F(k_1,\ldots,k_r;n)&\longrightarrow&\F(k_1,\ldots,k_l;n)\\
&&\{V_i\}_{i=1,\ldots,r}&\longmapsto&\{V_i\}_{i=1,\ldots,l}
\end{array}$$
The fiber of $\pi$ over a general point in $\F(k_1,\ldots,k_l;n)$ is isomorphic to $\F(k_{l+1}-k_l-1,\ldots,k_r-k_l-1;n-k_l-1)$. Now let $p_1,\ldots,p_{h}\in \F(k_1,\ldots,k_l;n)$ be general points, and $T_{p_i}\F(k_1,\ldots,k_l;n)$ be the tangent space at $p_i$. Then, we have 
$$T_{\pi^{-1}(p_i)}\F(k_1,\ldots,k_r;n)= \left\langle T_{p_i}\F(k_1,\ldots,k_l;n), T_{\pi^{-1}(p_i)}\F(k_{l+1}-k_l,\ldots,k_r-k_l;n-k_l)\right\rangle$$
and $T_{p_i}\F(k_1,\ldots,k_l;n)\cap T_{\pi^{-1}(p_i)}\F(k_{l+1}-k_l,\ldots,k_r-k_l;n-k_l) = \emptyset$.

Now, observe that if $T_{\pi^{-1}(p_i)}\F(k_1,\ldots,k_r;n)\cap T_{\pi^{-1}(p_j)}\F(k_1,\ldots,k_r;n)\neq \emptyset$ then 
$$\dim\left\langle T_{\pi^{-1}(p_j)}\F(k_1,\ldots,k_r;n)\:;\:j=1,\ldots,h\right\rangle\leq h\dim \F(k_1,\ldots,k_l;n)+h-2$$
Since $T_{\pi^{-1}(p_i)}\F(k_{l+1}-k_l-1,\ldots,k_r-k_l-1;n-k_l-1)$ is contracted by $\pi$ for any $j=1,\ldots,h$ we have that 
$$
\begin{array}{ccl}
\dim \pi(T)&\leq& h\dim \F(k_1,\ldots,k_r;n)+h-2-h\dim\F(k_{l+1}-k_l,\ldots,k_r-k_l;n-k_l)\\
&=&h\dim \F(k_1,\ldots,k_l;n)+h-2
\end{array}
$$ 
where $T=\left\langle T_{\pi^{-1}(p_i)}\F(k_1,\ldots,k_r;n)\:;\:i=1,\ldots,h\right\rangle$.

In particular, by Terracini's lemma \cite{Te12} we have that if $\F(k_1,\ldots,k_l;n)$ is not $h$-defective, then $\F(k_1,\ldots,k_r;n)$ is not $h$-defective.
\end{Remark}

\begin{thm}\label{main}
Consider a flag variety $\F(k_1,\ldots,k_r;n)$ with $n<2k_r+1$. Assume that $n\geq 2k_{j}+1$ for some index $j$ and let $l$ be the maximum among these j's. Then, for 
$$h\leq\left(\frac{n+1}{k_l+1}\right)^{\lfloor \log_2(\sum_{j=1}^l k_j+l-1)\rfloor}$$
$\F(k_1,\ldots,k_r;n)$ is not $(h+1)$-defective.
\end{thm}
\begin{proof}
It is an immediate consequence of Theorem \ref{Bound_Flags} and Remark \ref{reduction}.
\end{proof}

\subsection{On identifiability of products of Grassmannians and flag varieties}
Let $X\subset\mathbb{P}^N$ be an irreducible non-degenerated variety. A point $p\in\mathbb{P}^N$ is said to be $h$-identifiable, with respect to $X$, if it lies on a unique $(h-1)$-plane $h$-secant to $X$. Furthermore, $X$ is said to be $h$-identifiable if a general point of $\mathbb{S}ec_h(X)$ is $h$-identifiable.

Now, we combine our bounds on non-secant defectivity of products of Grassmannians and flag varieties and \cite[Theorem 3]{CM19} to get the following. 

\begin{Corollary}\label{CorId}
Consider the product of Grassmannians $\prod_{i=1}^r\mathbb{G}(k_i,n)$. Assume that $2\prod_{i=1}^r(k_i+1)(n-k_i)-1 \leq \left(\frac{n+1}{k_r+1}\right)^{\lfloor\log_2(\sum k_i+r-1)\rfloor}$. Then, $\prod_{i=1}^r\mathbb{G}(k_i,n)$ is $h$-identifiable for $h \leq \left(\frac{n+1}{k_r+1}\right)^{\lfloor\log_2(\sum k_i+r-1)\rfloor}$.

Furthermore, let us suppose that $n\geq 2k_{j}+1$ for some index $j$ and consider $l$ the maximum among these j's. Assume that $2((k_1+1)(n-k_1)+\sum_{j=2}^i(n-k_j)(k_j-k_{j-1}))-1\leq \left(\frac{n+1}{k_l+1}\right)^{\lfloor \log_2(\sum_{j=1}^l k_j+l-1)\rfloor}$. Then $\F(k_1,\ldots,k_r;n)$ is $h$-identifiable for $h\leq\left(\frac{n+1}{k_l+1}\right)^{\lfloor \log_2(\sum_{j=1}^l k_j+l-1)\rfloor}$.
\end{Corollary}
\begin{proof}
It is enough to apply Corollary \ref{main_cor}, Theorem \ref{main} and \cite[Theorem 3]{CM19}.
\end{proof}

\section{On the chordal variety of $\F(0,k;n)$}\label{sec4}
In this section we consider particularly flag varieties parametrizing chains of type $p\in H^k\subset\mathbb{P}^n$.

\begin{Proposition}
Let us consider the flag variety $\F(0,k;n)\subset\mathbb{P}(\Gamma)\subset\p^N$, where $0< k< n$. Then, $\mathbb{S}ec_2\F(0,k;n)$ has always the expected dimension except when $k=n-1$, in this case $\F(0,n-1;n)$ is $2$-defective with $2$-defect $\delta_2(\F(0,n-1;n)) = 1$.
\end{Proposition}
\begin{proof}
Let $p,q\in\F(0,k;n)$ be two general points, without lose the generality we can assume that $p=e_{0,\{0,\ldots,k\}}=e_{0,I_0}$ and $q=e_{n,\{n-k,\ldots,n\}}=e_{n,I_1}$.

Now, Proposition \ref{well-behaivior} yields that
$$\begin{array}{ccl}
T_{e_{0,I_0}}\F(0,k;n)&=&\langle e_{i,I}\:|\:d((i,I),(0,I_0))\leq 1\rangle\cap \mathbb{P}(\Gamma)\\
\end{array}$$
and
$$\begin{array}{ccl}
T_{e_{n,I_1}}\F(0,k;n)&=&\langle e_{i,I}\:|\:d((i,I),(n,I_n))\leq 1\rangle\cap \mathbb{P}(\Gamma)\\
\end{array}$$
Note that $d((i,I),(0,I_0))= 1$ if and only if either $i\neq 0$ and $I=I_0$ or $i=0$ and $|I\cap I_0|=k$. Similarly, $d((i,I),(n,I_1))=1$ if and only if either $i\neq n$ and $I=I_1$ or $i=n$ and $|I\cap I_1|=k$. Therefore, since $n\neq 0$ and $I_1\neq I_0$ we have that $e_{i,I}\in\{ e_{i,I}\:|\:d((i,I),(0,I_0))\leq 1\}\cap\{ e_{i,I}\:|\:d((i,I),(n,I_1))\leq 1\}$ if and only if either $I=I_0$ and $i=n$ or $I=I_1$ and $i=0$.

Now, assume that $I=I_0$ and $i=n$, this is $e_{i,I}\in T_{e_{0,I_0}}\F(0,k;n)\cap T_{e_{n,I_1}}\F(0,k;n)$, in particular we have $|I\cap I_1|=|I_0\cap I_1|=k$ and hence $\{1,\ldots,k\}\subset I_1$ once $0\notin I_1$. So we must have $k=n-1$. Similarly, if $I=I_1\text{ and }i=0$ we conclude that $k=n-1$.

Therefore, if $k<n-1$, we get 
$$\{ e_{i,I}\:|\:d((i,I),(0,I_0))\leq 1\}\cap\{ e_{i,I}\:|\:d((i,I),(n,I_1))\leq 1\}=\emptyset$$

and hence 
$$\{ e_{i,I}\:|\:d((i,I),(0,I_0))\leq 1\}\cap\{ e_{i,I}\:|\:d((i,I),(n,I_1))\leq 1\}\cap\mathbb{P}(\Gamma)=\emptyset$$
which implies that 
$$\dim\left\langle T_{e_{0,I_0}}\F(0,k;n),T_{e_{n,I_1}}\F(0,k;n)\right\rangle=2\dim\F(0,k;n)+1$$
So, Terracini's lemma \cite{Te12} yields that $\mathbb{S}ec_2\F(0,k;n)$ has the expected dimension whenever $k<n-1$.

Now, assume that $k=n-1$. In this case we have 
$$\{ e_{i,I}\:|\:d((i,I),(0,I_0))\leq 1\}\cap\{ e_{i,I}\:|\:d((i,I),(n,I_1))\leq 1\}=\{e_{0,\{1,\ldots,n\}},e_{n,\{0,\ldots,n-1\}}\}$$

Furthermore, $\F(0,n-1;n)$ is the hypersurface cutting out in $\p^n\times \p^{n*}$ by 
$$\displaystyle\sum_{i=0}^n(-1)^iZ_{i,I_n\setminus\{i\}}=0$$ 
where $I_n=\{0,\ldots,n\}$.

Therefore, we get that $T_{e_{0,I_0}}\F(0,n-1;n)=\langle e_{i,I}\:|\:d((i,I),(0,I_0))\leq 1\rangle\cap \mathbb{P}(\Gamma)$ is given by
$$\left\langle e_{0,\{1,\ldots,n\}}+(-1)^{n+1}e_{n,\{0,\ldots,n-1\}}\:;\:e_{i,I}\:|\:d((i,I),(0,I_0))\leq 1\text{ and }i,I\neq\left\{\begin{array}{l}
0,\{1,\ldots,n\}\\
n,\{1,\ldots,n-1\}
\end{array}\right.\right\rangle$$
and $T_{e_{n,I_1}}\F(0,n-1;n)=\langle e_{i,I}\:|\:d((i,I),(n,I_1))\leq 1\rangle\cap \mathbb{P}(\Gamma)$ is given by
$$\left\langle e_{0,\{1,\ldots,n\}}+(-1)^{n+1}e_{n,\{0,\ldots,n-1\}}\:;\:e_{i,I}\:|\:d((i,I),(n,I_1))\leq 1\text{ and }i,I\neq\left\{\begin{array}{l}
0,\{1,\ldots,n\}\\
n,\{1,\ldots,n-1\}
\end{array}\right.\right\rangle$$
Therefore, 
$$\dim\left\langle T_{e_{0,I_0}}\F(0,n-1;n),T_{e_{n,I_1}}\F(0,n-1;n)\right\rangle=2\dim\F(0,n-1;n)<\expd\mathbb{S}ec_2\F(0,n-1;n)$$
Finally, since $\expd\mathbb{S}ec_2\F(0,k;n)=2\dim\F(0,n-1;n)+1$ we have that $\F(0,n-1,n)$ is $2$-defective with 2-defect $\delta_2(\F(0,n-1;n)) = 1$.
\end{proof}

\bibliographystyle{amsalpha}
\bibliography{Biblio}
\end{document}